\documentclass[11pt,reqno]{amsart}

\usepackage[T1]{fontenc}
\usepackage[utf8]{inputenc}
\usepackage{lmodern}
\usepackage{microtype}
\usepackage[a4paper,margin=1in]{geometry}

\usepackage{amsmath,amssymb,amsfonts,amsthm}
\usepackage{mathtools}
\usepackage{bm}
\usepackage{mathrsfs}

\usepackage{enumitem}
\usepackage{xcolor}
\usepackage{aliascnt}
\usepackage[
    colorlinks=true,
    linkcolor=blue,
    citecolor=blue,
    urlcolor=blue
]{hyperref}
\usepackage[
    nameinlink,
    capitalise,
    noabbrev
]{cleveref}

\allowdisplaybreaks
\numberwithin{equation}{section}
\setlist[itemize]{topsep=3pt,itemsep=2pt,parsep=0pt}
\setlist[enumerate]{topsep=3pt,itemsep=2pt,parsep=0pt}

\theoremstyle{plain}
\newtheorem{theorem}{Theorem}[section]

\newaliascnt{lemma}{theorem}
\newtheorem{lemma}[lemma]{Lemma}
\aliascntresetthe{lemma}

\newaliascnt{proposition}{theorem}
\newtheorem{proposition}[proposition]{Proposition}
\aliascntresetthe{proposition}

\newaliascnt{corollary}{theorem}
\newtheorem{corollary}[corollary]{Corollary}
\aliascntresetthe{corollary}

\theoremstyle{definition}
\newaliascnt{definition}{theorem}
\newtheorem{definition}[definition]{Definition}
\aliascntresetthe{definition}

\newaliascnt{assumption}{theorem}
\newtheorem{assumption}[assumption]{Assumption}
\aliascntresetthe{assumption}

\newaliascnt{example}{theorem}
\newtheorem{example}[example]{Example}
\aliascntresetthe{example}

\theoremstyle{remark}
\newaliascnt{remark}{theorem}
\newtheorem{remark}[remark]{Remark}
\aliascntresetthe{remark}

\crefname{theorem}{theorem}{theorems}
\Crefname{theorem}{Theorem}{Theorems}
\crefname{lemma}{lemma}{lemmas}
\Crefname{lemma}{Lemma}{Lemmas}
\crefname{proposition}{proposition}{propositions}
\Crefname{proposition}{Proposition}{Propositions}
\crefname{corollary}{corollary}{corollaries}
\Crefname{corollary}{Corollary}{Corollaries}
\crefname{definition}{definition}{definitions}
\Crefname{definition}{Definition}{Definitions}
\crefname{assumption}{assumption}{assumptions}
\Crefname{assumption}{Assumption}{Assumptions}
\crefname{remark}{remark}{remarks}
\Crefname{remark}{Remark}{Remarks}
\crefname{example}{example}{examples}
\Crefname{example}{Example}{Examples}
\crefname{section}{section}{sections}
\Crefname{section}{Section}{Sections}

\newcommand{\RR}{\mathbb{R}}

\newcommand{\NN}{\mathbb{N}}

\newcommand{\Om}{\Omega}
\newcommand{\Hilb}{H}
\newcommand{\EnergySpace}{V}
\newcommand{\StrongSpace}{E}
\newcommand{\Sphere}{\mathcal{M}}

\newcommand{\Bop}{\mathcal{B}}
\newcommand{\Aop}{\mathcal{A}}
\newcommand{\Ppoly}{\mathscr{P}}
\newcommand{\Proj}[1]{\Pi_{#1}}
\newcommand{\EqSetLin}{\mathcal{E}_{\mathrm{eq}}^{\mathrm{lin}}}
\newcommand{\Semigroup}[1]{\mathrm{e}^{-#1\Aop}}
\newcommand{\Tan}[1]{T_{#1}\Sphere}

\DeclareMathOperator{\Dom}{\mathcal{D}}

\DeclareMathOperator{\Ker}{Ker}
\DeclareMathOperator{\Span}{span}

\DeclareMathOperator{\spec}{\sigma}

\DeclarePairedDelimiter{\norm}{\lVert}{\rVert}

\DeclarePairedDelimiterX{\inner}[2]{\langle}{\rangle}{#1,#2}

\newcommand{\dd}{\,\mathrm{d}}
\newcommand{\pt}{\partial_t}
\newcommand{\degree}{m}
\newcommand{\shift}{\kappa}
\newcommand{\rhoSH}{\rho}

\title[Polynomial Spectral Selection]
{Spectral Selection in Sphere-Constrained Flows Generated by Polynomials of the Dirichlet Laplacian}

\author{Javed Hussain}
\address{Department of Mathematics, Sukkur IBA University,
Airport Road, Sukkur 65200, Pakistan}
\email{javed.brohi@iba-suk.edu.pk}

\begin{document}

\begin{abstract}
Let \(\Om\subset\RR^d\) be a bounded smooth domain and let
\(\Bop=-\Delta_D\) be the positive Dirichlet Laplacian on
\(\Hilb=L^2(\Om)\).  For a real polynomial \(\Ppoly\) with positive
leading coefficient, we study the constrained linear equation
\[
    u_t=-\Proj{u}\Ppoly(\Bop)u,
    \qquad \norm{u(0)}_{\Hilb}=1.
\]
Its solution is the normalized semigroup orbit
\[
    u(t)=
    \frac{e^{-t\Ppoly(\Bop)}u_0}
         {\norm{e^{-t\Ppoly(\Bop)}u_0}_{\Hilb}}.
\]
The active spectral support is preserved, and the trajectory converges to the
normalized projection of \(u_0\) onto the active eigenspaces for which
\(\Ppoly(\lambda_j)\) is minimal.  The next active polynomial spectral value
gives the exponential rate.  For every \(\theta\geq0\) and \(\tau>0\), the
same rate holds in \(\Dom(\Aop^\theta)\) for \(t\geq\tau\), even when the
initial datum has no fractional regularity.  We also show that a finite set of
Dirichlet levels can be prescribed as the global minimizing set of a
polynomial, and that an isolated selected set is stable under sufficiently
small polynomial perturbations.  For \(\Ppoly(s)=s^m\) the lowest active
Dirichlet level is selected.  For \(\Ppoly(s)=(s-\rho)^2\), selection is by
distance from \(\rho\), and cross-level degeneracy occurs only at Dirichlet
midpoints.
\end{abstract}

\subjclass[2020]{Primary 35B40; Secondary 35P05, 47D06, 58D25, 37L05}

\keywords{sphere-constrained flow, polynomial of the Dirichlet Laplacian,
normalized semigroup, Rayleigh-quotient flow, spectral selection,
active spectral support, spectral gap, polynomial perturbation,
polyharmonic flow, Swift--Hohenberg operator}

\maketitle

\section{Introduction}
\label{sec:introduction}

For the normalized heat flow on the \(L^2\)-sphere, the Dirichlet spectral
ordering determines the asymptotic state.  Replacing the Dirichlet Laplacian
by a polynomial of that operator preserves the eigenspaces but changes their
dynamical ordering from \(\lambda_j\) to \(\Ppoly(\lambda_j)\).  The
polynomial therefore affects the constrained dynamics through its values on
the Dirichlet spectrum, not merely through its degree.  When \(\Ppoly\) is
non-monotone, an interior Dirichlet level may be asymptotically preferred to
the lowest active level, and several distinct spectral levels may be assigned
the same minimal polynomial value.

Let \(\Om\subset\RR^d\) be a bounded smooth domain, set
\[
    \Hilb=L^2(\Om),
    \qquad
    \Bop=-\Delta_D,
\]
and write
\[
    \Bop e_j=\lambda_je_j,
    \qquad
    0<\lambda_1\leq\lambda_2\leq\cdots,
    \qquad
    \lambda_j\to\infty,
\]
where the eigenvalues are repeated according to multiplicity and
\(\{e_j\}_{j\geq1}\) is an orthonormal basis of \(\Hilb\).  Let
\[
    \Ppoly(s)
    =a_{\degree}s^{\degree}
     +a_{\degree-1}s^{\degree-1}
     +\cdots+a_1s+a_0,
    \qquad
    a_{\degree}>0.
\]
The spectral realization satisfies
\[
    \Ppoly(\Bop)e_j=\Ppoly(\lambda_j)e_j.
\]
On the unit sphere
\[
    \Sphere:=\{u\in\Hilb:\norm{u}_{\Hilb}=1\},
\]
the tangent projection is
\[
    \Proj{u}v
    =v-\inner{v}{u}_{\Hilb}u,
    \qquad u\in\Sphere,
\]
and we study
\begin{equation}
\label{eq:intro-main-flow}
    u_t=-\Proj{u}\Ppoly(\Bop)u,
    \qquad
    u(0)=u_0\in\Sphere.
\end{equation}

Mahony and Absil \cite{MahonyAbsil2003} studied the continuous-time
Rayleigh quotient flow on a finite-dimensional sphere.  Normalized gradient
flows also appear in eigenvalue and constrained minimization procedures; see
Bao and Du \cite{BaoDu2004}, Liu and Luo \cite{LiuLuo2010}, and
Besse, Duboscq and Le Coz \cite{BesseDuboscqLeCoz2022}.  Infinite-dimensional
constrained heat flows are considered in
\cite{Rybka2006,CaffarelliLin2009,Hussain2020IJMCS,BrzezniakHussain2024}.
For the operator studied here, the Dirichlet eigenspaces are fixed but their
ordering is changed by
\[
    \lambda_j\longmapsto \Ppoly(\lambda_j).
\]
The results below describe the selected active eigenspace, its rate of
attraction, and the effect of changing the polynomial.

A constant shift is useful because \(\Ppoly(\Bop)\) need not be positive.
Choose \(\shift>0\) so that
\[
    \Aop:=\Ppoly(\Bop)+\shift I\geq I.
\]
Since
\[
    \Proj{u}(\shift u)=0,
    \qquad u\in\Sphere,
\]
the shift leaves \eqref{eq:intro-main-flow} unchanged.  It supplies the
positive Hilbert scale
\[
    \Dom(\Aop^\theta)=\Dom(\Bop^{\degree\theta}),
    \qquad \theta\geq0,
\]
with equivalent norms.  On smooth domains,
\[
    \Dom(\Bop^\degree)
    =
    \left\{
        u\in H^{2\degree}(\Om):
        \Delta^ku|_{\partial\Om}=0,
        \ k=0,\ldots,\degree-1
    \right\},
\]
so the spectral realization corresponds to Navier-type boundary conditions.
The smooth boundary is needed here only for this classical domain
identification; the spectral arguments use positivity, self-adjointness and
compact resolvent.

For \(u_0=\sum_{j\geq1}c_je_j\), the solution of
\eqref{eq:intro-main-flow} is
\begin{equation}
\label{eq:intro-exact-formula}
    u(t)
    =
    \frac{e^{-t\Ppoly(\Bop)}u_0}
         {\norm{e^{-t\Ppoly(\Bop)}u_0}_{\Hilb}}
    =
    \frac{
      \displaystyle\sum_{j=1}^\infty
      c_je^{-t\Ppoly(\lambda_j)}e_j
    }{
      \displaystyle
      \left(
        \sum_{j=1}^\infty
        |c_j|^2e^{-2t\Ppoly(\lambda_j)}
      \right)^{1/2}
    }.
\end{equation}
Thus no Dirichlet component can be created or removed at finite time.  If
\[
    \Sigma_P(u_0)
    :=\{\Ppoly(\lambda_j):c_j\neq0\},
    \qquad
    \mu_*(u_0):=\min\Sigma_P(u_0),
\]
and \(\mathsf E_*\) denotes the orthogonal projection onto
\[
    \Ker\bigl(\Ppoly(\Bop)-\mu_*(u_0)I\bigr),
\]
then the main selection theorem gives
\begin{equation}
\label{eq:intro-selection-limit}
    u(t)
    \longrightarrow
    \frac{\mathsf E_*u_0}
         {\norm{\mathsf E_*u_0}_{\Hilb}}
    \qquad\text{in }\Hilb.
\end{equation}
If another active polynomial value is present, the convergence is
exponential with exponent
\[
    \delta_*(u_0)
    =
    \min\{\mu-\mu_*(u_0):
          \mu\in\Sigma_P(u_0),\ \mu>\mu_*(u_0)\}.
\]
The same exponent governs convergence in every fractional domain for positive
times, without any fractional regularity assumption on \(u_0\).

Two further consequences use the polynomial structure in a way that is absent
from the ordinary heat flow.  First, any prescribed finite set of distinct
Dirichlet levels \(J=\{\nu_1,\ldots,\nu_r\}\) can be made the global
minimizing set by taking
\[
    \Ppoly_J(s)=\prod_{k=1}^r(s-\nu_k)^2.
\]
Second, if a set of active levels is isolated by a positive polynomial gap,
then it remains separated from all other active levels under sufficiently
small polynomial perturbations.  A perturbation can split a degenerate
selected set, and the splitting is determined by the perturbing polynomial
restricted to that finite set.

The polyharmonic and Swift--Hohenberg polynomials illustrate two different
orderings.  For \(\Ppoly(s)=s^\degree\), the polynomial is strictly
increasing on the positive Dirichlet spectrum, so the lowest active level is
selected.  For
\[
    \Ppoly_\rho(s)=(s-\rho)^2,
\]
the selected active levels are those closest to \(\rho\).  This reflects the
preferred-scale structure associated with the Swift--Hohenberg operator
\cite{SwiftHohenberg1977,CrossHohenberg1993,PeletierRottschafer2004,
PeletierWilliams2007}.  Distinct levels can be selected simultaneously only
when \(\rho\) is their midpoint; consequently cross-level degeneracy occurs
on a countable exceptional set of parameter values.

The selection proof starts from \eqref{eq:intro-exact-formula}: after the
smallest active exponential is factored out, the orthogonal remainder decays
at the active spectral gap.  Positive-time smoothing gives the corresponding
fractional-domain estimates.

\Cref{sec:operator-setting} fixes the operator realization and its Hilbert
scale.  The constrained flow and the selection theorem are proved in
\Cref{sec:linear-selection}.  Spectral design and perturbations are treated in
\Cref{sec:design-robustness}, followed by the polyharmonic and
Swift--Hohenberg cases in \Cref{sec:examples}.

\section{Polynomial Dirichlet operators and the functional setting}
\label{sec:operator-setting}

We use the spectral realization of \(\Ppoly(\Bop)\).  Its domain is fixed by
the Dirichlet spectral calculus, and a constant shift gives a positive
operator without changing the projected flow.  The resulting fractional
powers provide the Hilbert scale used below.

\subsection{The Dirichlet Laplacian and its polynomial functional calculus}
\label{subsec:polynomial-functional-calculus}

We begin with the spectral realization of the positive Dirichlet Laplacian and
fix the structural assumptions on the polynomial.

\begin{assumption}
\label{ass:domain-polynomial}
Let \(\Om\subset\RR^d\), \(d\geq 1\), be a bounded connected domain with
\(C^\infty\)-boundary.  We work in the real Hilbert space
\[
    \Hilb:=L^2(\Om),
    \qquad
    \inner{u}{v}_{\Hilb}
    :=
    \int_{\Om}u(x)v(x)\,\dd x.
\]
Let
\[
    \Bop=-\Delta_D,
    \qquad
    \Dom(\Bop)=H^2(\Om)\cap H_0^1(\Om),
\]
be the positive Dirichlet Laplacian.  Moreover, let
\begin{equation}
    \Ppoly(s)
    =
    a_{\degree}s^{\degree}
    +a_{\degree-1}s^{\degree-1}
    +\cdots+a_1s+a_0,
    \qquad
    \degree\in\NN,
    \label{eq:polynomial-P}
\end{equation}
be a real polynomial with
\[
    a_{\degree}>0.
\]
\end{assumption}

The \(C^\infty\)-assumption will be used when the spectral domains are
identified with classical Navier domains.  Until that point, only positivity,
self-adjointness and compact resolvent of \(\Bop\) are needed.

\begin{lemma}
\label{lem:dirichlet-spectral}
Under \Cref{ass:domain-polynomial}, the operator \(\Bop\) is positive,
self-adjoint, and has compact inverse on \(\Hilb\).  Moreover,
\begin{equation}
    \Dom(\Bop^{1/2})=H_0^1(\Om),
    \qquad
    \norm{\Bop^{1/2}u}_{\Hilb}
    =
    \norm{\nabla u}_{L^2(\Om)}.
    \label{eq:form-domain-B}
\end{equation}
Consequently, there exist eigenvalues
\[
    0<\lambda_1\leq\lambda_2\leq\cdots,
    \qquad
    \lambda_j\longrightarrow\infty,
\]
repeated according to multiplicity, and an orthonormal basis
\(\{e_j\}_{j\geq1}\) of \(\Hilb\) such that
\begin{equation}
    \Bop e_j=\lambda_j e_j,
    \qquad j\geq1.
    \label{eq:B-eigenbasis}
\end{equation}
\end{lemma}

\begin{proof}
Consider the symmetric bilinear form
\[
    \mathfrak b(u,v)
    :=
    \int_{\Om}\nabla u(x)\cdot\nabla v(x)\,\dd x,
    \qquad
    u,v\in H_0^1(\Om).
\]
The space \(H_0^1(\Om)\) is dense in \(L^2(\Om)\), and Poincar\'e's
inequality gives a constant \(c_{\Om}>0\) such that
\[
    \mathfrak b(u,u)
    =
    \norm{\nabla u}_{L^2(\Om)}^2
    \geq
    c_{\Om}\norm{u}_{L^2(\Om)}^2,
    \qquad
    u\in H_0^1(\Om).
\]
Thus \(\mathfrak b\) is densely defined, closed, symmetric, and coercive.
The representation theorem for closed symmetric forms therefore determines a
positive self-adjoint operator \(\widetilde{\Bop}\) in \(\Hilb\) satisfying
\[
    \Dom(\widetilde{\Bop}^{1/2})=H_0^1(\Om)
\]
and
\[
    \inner{\widetilde{\Bop}^{1/2}u}
           {\widetilde{\Bop}^{1/2}v}_{\Hilb}
    =
    \mathfrak b(u,v),
    \qquad
    u,v\in H_0^1(\Om).
\]
On a smooth bounded domain, elliptic regularity for the homogeneous Dirichlet
problem identifies
\[
    \Dom(\widetilde{\Bop})
    =
    H^2(\Om)\cap H_0^1(\Om),
    \qquad
    \widetilde{\Bop}u=-\Delta u.
\]
Hence \(\widetilde{\Bop}=\Bop\), and \eqref{eq:form-domain-B} follows.

Coercivity implies that \(0\) belongs to the resolvent set of \(\Bop\).
Furthermore,
\[
    \Bop^{-1}:\Hilb\longrightarrow H^2(\Om)\cap H_0^1(\Om)
\]
is bounded.  Since the embedding \(H^2(\Om)\hookrightarrow L^2(\Om)\) is
compact on a bounded smooth domain, \(\Bop^{-1}\) is compact as an operator
from \(\Hilb\) into itself.  The spectral theorem for positive self-adjoint
operators with compact inverse then gives an orthonormal basis of eigenvectors
and a sequence of positive eigenvalues tending to infinity, which proves
\eqref{eq:B-eigenbasis}.
\end{proof}

For
\[
    u=\sum_{j=1}^{\infty}u_j e_j,
    \qquad
    u_j:=\inner{u}{e_j}_{\Hilb},
\]
we define the polynomial of the Dirichlet Laplacian by
\begin{equation}
    \Ppoly(\Bop)u
    :=
    \sum_{j=1}^{\infty}
    \Ppoly(\lambda_j)u_j e_j
    \label{eq:P-B-definition}
\end{equation}
on the spectral domain
\begin{equation}
    \Dom(\Ppoly(\Bop))
    :=
    \left\{
        u\in\Hilb:
        \sum_{j=1}^{\infty}
        |\Ppoly(\lambda_j)|^2|u_j|^2<\infty
    \right\}.
    \label{eq:P-B-domain}
\end{equation}
The highest-order term of the polynomial determines this domain.

\begin{proposition}
\label{prop:domain-PB}
Under \Cref{ass:domain-polynomial},
\begin{equation}
    \Dom(\Ppoly(\Bop))=\Dom(\Bop^{\degree}),
    \label{eq:domain-PB}
\end{equation}
with equivalent graph norms.  Moreover, \(\Ppoly(\Bop)\) is self-adjoint on
\(\Hilb\).
\end{proposition}

\begin{proof}
Since \(a_{\degree}>0\),
\[
    \lim_{s\to\infty}
    \frac{\Ppoly(s)}{s^{\degree}}
    =
    a_{\degree}.
\]
Hence there exists \(R>\lambda_1\) such that
\[
    \frac{a_{\degree}}{2}s^{\degree}
    \leq
    |\Ppoly(s)|
    \leq
    2a_{\degree}s^{\degree},
    \qquad
    s\geq R.
\]
On the compact interval \([\lambda_1,R]\), both
\(1+|\Ppoly(s)|^2\) and \(1+s^{2\degree}\) are continuous and strictly
positive.  Consequently, there exist constants \(c_P,C_P>0\) such that
\begin{equation}
    c_P\bigl(1+s^{2\degree}\bigr)
    \leq
    1+|\Ppoly(s)|^2
    \leq
    C_P\bigl(1+s^{2\degree}\bigr),
    \qquad
    s\geq\lambda_1.
    \label{eq:P-growth-equivalence}
\end{equation}
If \(u=\sum_{j\geq1}u_je_j\in\Hilb\), then
\[
    u\in\Dom(\Ppoly(\Bop))
\]
if and only if
\[
    \sum_{j=1}^{\infty}
    |\Ppoly(\lambda_j)|^2|u_j|^2<\infty.
\]
Since \(u\in\Hilb\) already implies
\(\sum_{j\geq1}|u_j|^2<\infty\), the two-sided estimate
\eqref{eq:P-growth-equivalence} shows that the preceding condition is
equivalent to
\[
    \sum_{j=1}^{\infty}
    \lambda_j^{2\degree}|u_j|^2<\infty.
\]
The latter condition is precisely \(u\in\Dom(\Bop^{\degree})\).
This proves \eqref{eq:domain-PB} and the equivalence of the corresponding
graph norms.

It remains to check self-adjointness.  The operator is symmetric because the
multipliers \(\Ppoly(\lambda_j)\) are real.  Let
\(v\in\Dom(\Ppoly(\Bop)^*)\).  Then there is \(g\in\Hilb\) such that
\[
    \inner{\Ppoly(\Bop)u}{v}_{\Hilb}
    =\inner{u}{g}_{\Hilb}
    \qquad
    \text{for every }u\in\Dom(\Ppoly(\Bop)).
\]
Taking \(u=e_j\) gives
\[
    \Ppoly(\lambda_j)\inner{e_j}{v}_{\Hilb}
    =\inner{e_j}{g}_{\Hilb}.
\]
Hence
\[
    \sum_{j=1}^\infty
    |\Ppoly(\lambda_j)|^2
    |\inner{v}{e_j}_{\Hilb}|^2
    =\norm{g}_{\Hilb}^2<\infty.
\]
Thus \(v\in\Dom(\Ppoly(\Bop))\) and
\(\Ppoly(\Bop)v=g\).  Therefore
\(\Dom(\Ppoly(\Bop)^*)\subset\Dom(\Ppoly(\Bop))\); the reverse inclusion
follows from symmetry.  Hence \(\Ppoly(\Bop)^*=\Ppoly(\Bop)\).
\end{proof}

The polynomial \(\Ppoly\) need not be positive on the spectrum of \(\Bop\).
For the evolution theory it is convenient to replace \(\Ppoly(\Bop)\) by a
strictly positive operator.  The sphere projection will later show that this
modification does not alter the constrained equation.

Since
\[
    \Ppoly(s)\longrightarrow+\infty
    \qquad\text{as }s\to\infty,
\]
the quantity
\begin{equation}
    \mu_P
    :=
    \min_{s\in[\lambda_1,\infty)}\Ppoly(s)
    \label{eq:mu-P}
\end{equation}
is finite.  Fix throughout the paper
\begin{equation}
    \shift>\max\{0,1-\mu_P\}
    \label{eq:kappa-choice}
\end{equation}
and define
\begin{equation}
    \Aop
    :=
    \Ppoly(\Bop)+\shift I.
    \label{eq:A-definition}
\end{equation}
If
\[
    Q(s):=\Ppoly(s)+\shift,
\]
then
\begin{equation}
    Q(s)\geq1,
    \qquad
    s\geq\lambda_1.
    \label{eq:Q-positive}
\end{equation}

\begin{lemma}
\label{lem:shift-comparison}
There exist constants \(c_A,C_A>0\) such that
\begin{equation}
    c_A(1+s^{\degree})
    \leq
    Q(s)
    \leq
    C_A(1+s^{\degree}),
    \qquad
    s\geq\lambda_1.
    \label{eq:Q-comparison}
\end{equation}
Consequently, \(\Aop\) is strictly positive and self-adjoint,
\begin{equation}
    \Dom(\Aop)=\Dom(\Bop^{\degree}),
    \label{eq:domain-A}
\end{equation}
and \(\Aop^{-1}\) is compact on \(\Hilb\).
\end{lemma}

\begin{proof}
The positivity follows directly from \eqref{eq:kappa-choice} and
\eqref{eq:mu-P}.  Moreover,
\[
    \frac{Q(s)}{s^{\degree}}
    =
    \frac{\Ppoly(s)}{s^{\degree}}
    +
    \frac{\shift}{s^{\degree}}
    \longrightarrow
    a_{\degree}>0
    \qquad\text{as }s\to\infty.
\]
Hence \eqref{eq:Q-comparison} holds for all sufficiently large \(s\).
After modifying the constants \(c_A\) and \(C_A\), the same estimate holds
on the remaining compact interval because \(Q\) is continuous and bounded
below there by \(1\).

Since \(\Aop\) and \(\Ppoly(\Bop)\) differ by the bounded operator
\(\shift I\), their domains coincide.  Thus \eqref{eq:domain-A} follows from
\Cref{prop:domain-PB}.  Self-adjointness follows from the self-adjointness of
\(\Ppoly(\Bop)\) and the bounded self-adjoint perturbation \(\shift I\).

Finally, with
\begin{equation}
    q_j:=Q(\lambda_j),
    \qquad j\geq1,
    \label{eq:qj-definition}
\end{equation}
we have
\[
    \Aop e_j=q_j e_j,
    \qquad
    q_j\longrightarrow\infty
\]
by \eqref{eq:Q-comparison}.  Hence
\[
    \Aop^{-1}u
    =
    \sum_{j=1}^{\infty}q_j^{-1}u_j e_j.
\]
The finite-rank operators
\[
    K_Nu
    :=
    \sum_{j=1}^{N}q_j^{-1}u_j e_j
\]
satisfy
\[
    \norm{\Aop^{-1}-K_N}_{\mathcal L(\Hilb)}
    =
    \sup_{j>N}q_j^{-1}
    \longrightarrow0.
\]
Thus \(\Aop^{-1}\) is the operator-norm limit of finite-rank operators and
is therefore compact.
\end{proof}

We now record explicitly why the shift introduced in
\eqref{eq:A-definition} has no effect on the sphere-constrained dynamics.

\begin{lemma}[Shift invariance under the tangent projection]
\label{lem:shift-invariance}
Let
\[
    \Sphere
    :=
    \{u\in\Hilb:\norm{u}_{\Hilb}=1\}
\]
and, for \(u\in\Sphere\), define
\begin{equation}
    \Proj{u}v
    :=
    v-\inner{v}{u}_{\Hilb}u,
    \qquad
    v\in\Hilb.
    \label{eq:tangent-projection}
\end{equation}
Then, for every \(u\in\Dom(\Bop^{\degree})\cap\Sphere\),
\begin{equation}
    \Proj{u}\Aop u
    =
    \Proj{u}\Ppoly(\Bop)u.
    \label{eq:shift-invariance}
\end{equation}
\end{lemma}

\begin{proof}
Since \(\norm{u}_{\Hilb}=1\),
\[
    \Proj{u}u
    =
    u-\inner{u}{u}_{\Hilb}u
    =
    0.
\]
Using \(\Aop=\Ppoly(\Bop)+\shift I\), linearity of \(\Proj{u}\) therefore
gives
\[
    \Proj{u}\Aop u
    =
    \Proj{u}\Ppoly(\Bop)u
    +
    \shift\Proj{u}u
    =
    \Proj{u}\Ppoly(\Bop)u,
\]
which proves \eqref{eq:shift-invariance}.
\end{proof}

\subsection{Fractional domains and the natural Hilbert scale}
\label{subsec:fractional-domains}

The positive operator \(\Aop\) determines the energy and strong-solution
spaces.  Their relation to the fractional powers of the Dirichlet Laplacian
follows directly from the spectral comparison \eqref{eq:Q-comparison}.

For \(\theta\geq0\), define
\[
    \Aop^\theta u
    :=
    \sum_{j=1}^{\infty}
    q_j^\theta u_j e_j
\]
with
\begin{equation}
    \Dom(\Aop^\theta)
    :=
    \left\{
        u\in\Hilb:
        \sum_{j=1}^{\infty}
        q_j^{2\theta}|u_j|^2<\infty
    \right\}.
    \label{eq:A-fractional-domain}
\end{equation}

\begin{proposition}
\label{prop:fractional-domain-equivalence}
For every \(\theta\geq0\),
\begin{equation}
    \Dom(\Aop^\theta)
    =
    \Dom(\Bop^{\degree\theta}),
    \label{eq:fractional-domain-equivalence}
\end{equation}
and the norms
\[
    u\longmapsto\norm{\Aop^\theta u}_{\Hilb}
    \qquad\text{and}\qquad
    u\longmapsto
    \norm{u}_{\Hilb}
    +
    \norm{\Bop^{\degree\theta}u}_{\Hilb}
\]
are equivalent on this common domain.
\end{proposition}

\begin{proof}
By \eqref{eq:Q-comparison},
\[
    c_A(1+\lambda_j^\degree)
    \leq
    q_j
    \leq
    C_A(1+\lambda_j^\degree)
\]
for every \(j\geq1\).  Raising this estimate to the power \(2\theta\)
gives constants \(c_\theta,C_\theta>0\) such that
\begin{equation}
    c_\theta
    \bigl(1+\lambda_j^{2\degree\theta}\bigr)
    \leq
    q_j^{2\theta}
    \leq
    C_\theta
    \bigl(1+\lambda_j^{2\degree\theta}\bigr),
    \qquad j\geq1.
    \label{eq:qj-fractional-comparison}
\end{equation}
Indeed, if \(r=2\theta\geq0\), then for \(x\geq0\) the elementary
inequalities for powers of sums give constants \(c_r,C_r>0\) such that
\[
    c_r(1+x^r)\leq(1+x)^r\leq C_r(1+x^r).
\]
Hence \((1+x)^{2\theta}\) and \(1+x^{2\theta}\) are equivalent uniformly
for \(x\geq0\).

Let \(u=\sum_{j\geq1}u_je_j\).  From
\eqref{eq:qj-fractional-comparison},
\[
    \sum_{j=1}^{\infty}
    q_j^{2\theta}|u_j|^2<\infty
\]
if and only if
\[
    \sum_{j=1}^{\infty}
    \bigl(1+\lambda_j^{2\degree\theta}\bigr)|u_j|^2<\infty.
\]
Since \(u\in\Hilb\), the latter condition is equivalent to
\[
    \sum_{j=1}^{\infty}
    \lambda_j^{2\degree\theta}|u_j|^2<\infty,
\]
which is precisely
\(u\in\Dom(\Bop^{\degree\theta})\).  The same two-sided estimate gives the
equivalence of norms.
\end{proof}

We shall use the notation
\begin{equation}
    \StrongSpace
    :=
    \Dom(\Aop)
    =
    \Dom(\Bop^\degree),
    \qquad
    \EnergySpace
    :=
    \Dom(\Aop^{1/2})
    =
    \Dom(\Bop^{\degree/2}).
    \label{eq:E-V-definition}
\end{equation}
Because \(\Aop\geq I\), the quantities
\begin{equation}
    \norm{u}_{\StrongSpace}
    :=
    \norm{\Aop u}_{\Hilb},
    \qquad
    u\in\StrongSpace,
    \label{eq:E-norm}
\end{equation}
and
\begin{equation}
    \norm{u}_{\EnergySpace}
    :=
    \norm{\Aop^{1/2}u}_{\Hilb},
    \qquad
    u\in\EnergySpace,
    \label{eq:V-norm}
\end{equation}
are equivalent to the corresponding graph norms.

\begin{corollary}
\label{cor:hilbert-triple}
The embeddings
\begin{equation}
    \StrongSpace
    \hookrightarrow
    \EnergySpace
    \hookrightarrow
    \Hilb
    \label{eq:hilbert-triple}
\end{equation}
are dense, continuous, and compact.
More generally, if \(0\leq\theta_1<\theta_2\), then
\begin{equation}
    \Dom(\Aop^{\theta_2})
    \Subset
    \Dom(\Aop^{\theta_1}).
    \label{eq:compact-fractional-embedding}
\end{equation}
\end{corollary}

\begin{proof}
Density follows because every finite linear combination of the eigenfunctions
\(\{e_j\}\) belongs to \(\Dom(\Aop^\theta)\) for every \(\theta\geq0\), and
such finite combinations are dense in \(\Hilb\).

Since \(q_j\geq1\),
\[
    q_j^{2\theta_1}
    \leq
    q_j^{2\theta_2},
\]
so
\[
    \norm{u}_{\Dom(\Aop^{\theta_1})}
    \leq
    \norm{u}_{\Dom(\Aop^{\theta_2})}.
\]
Thus the embeddings are continuous.

It remains to prove compactness.  Let \((u^k)_{k\geq1}\) be bounded in
\(\Dom(\Aop^{\theta_2})\); hence, for some \(M>0\),
\[
    \sum_{j=1}^{\infty}
    q_j^{2\theta_2}
    |u_j^k|^2
    \leq M^2,
    \qquad k\geq1.
\]
For \(N\in\NN\), let
\[
    \mathcal P_Nu
    :=
    \sum_{j=1}^{N}u_j e_j.
\]
Then
\[
\begin{aligned}
    \norm{(I-\mathcal P_N)u^k}_{\Dom(\Aop^{\theta_1})}^2
    &=
    \sum_{j>N}
    q_j^{2\theta_1}|u_j^k|^2
\\
    &=
    \sum_{j>N}
    q_j^{-2(\theta_2-\theta_1)}
    q_j^{2\theta_2}|u_j^k|^2
\\
    &\leq
    \left(
        \sup_{j>N}q_j^{-2(\theta_2-\theta_1)}
    \right)M^2.
\end{aligned}
\]
Since \(q_j\to\infty\), the right-hand side tends to zero as
\(N\to\infty\), uniformly in \(k\).  For fixed \(N\), the sequence
\((\mathcal P_Nu^k)_k\) lies in a bounded subset of a finite-dimensional
space and therefore possesses a convergent subsequence.  A diagonal argument,
combined with the uniform tail estimate above, yields a subsequence of
\((u^k)_k\) converging in \(\Dom(\Aop^{\theta_1})\).
This proves \eqref{eq:compact-fractional-embedding}.  Taking
\((\theta_2,\theta_1)=(1,\frac12)\) and
\((\frac12,0)\) gives \eqref{eq:hilbert-triple}.
\end{proof}

\begin{remark}[Boundary conditions encoded by the spectral realization]
\label{rem:boundary-conditions}
The spectral domain
\[
    \Dom(\Bop^{\degree})
\]
serves as the primary definition of the higher-order operator domain.  It should not
be replaced by the notation \(H_0^{2\degree}(\Om)\).  On a smooth domain,
elliptic regularity and induction show that
\begin{equation}
    \Dom(\Bop^\degree)
    =
    \left\{
        u\in H^{2\degree}(\Om):
        \Delta^k u\big|_{\partial\Om}=0,
        \quad
        k=0,\ldots,\degree-1
    \right\}.
    \label{eq:Bm-navier-domain}
\end{equation}
Indeed, if \(u\in\Dom(\Bop^\degree)\), then
\(\Bop^ku\in\Dom(\Bop)\) for \(k=0,\ldots,\degree-1\); hence each
\(\Bop^ku\) has zero Dirichlet trace, and repeated elliptic regularity gives
\(u\in H^{2\degree}(\Om)\).  Conversely, if
\(u\in H^{2\degree}(\Om)\) satisfies the boundary conditions in
\eqref{eq:Bm-navier-domain}, then
\((-\Delta)^ku\in H^2(\Om)\cap H_0^1(\Om)=\Dom(\Bop)\) for
\(k=0,\ldots,\degree-1\), and the recursive definition of operator powers
gives \(u\in\Dom(\Bop^\degree)\).

For example,
\[
    \Dom(\Bop)
    =
    H^2(\Om)\cap H_0^1(\Om),
\]
whereas
\[
    \Dom(\Bop^2)
    =
    \left\{
        u\in H^4(\Om):
        u|_{\partial\Om}=0,\ 
        \Delta u|_{\partial\Om}=0
    \right\}.
\]
Thus the fourth-order realization associated with \(\Bop^2\) carries Navier
boundary conditions rather than the clamped conditions
\(u=\partial_\nu u=0\).  No analogous identification
\(\Dom(\Bop^{\degree/2})=H_0^\degree(\Om)\) will be used; fractional
Dirichlet domains retain boundary conditions determined by the spectral
power.  For fractional Dirichlet domains see
\cite{FeffermanHajdukRobinson2022}; higher-order Navier realizations are
treated systematically in \cite{GazzolaGrunauSweers2010}.
\end{remark}

\subsection{Analytic semigroup estimates}
\label{subsec:semigroup-estimates}

The positivity and self-adjointness of \(\Aop\) yield an analytic semigroup
with estimates adapted exactly to the spaces \(\StrongSpace\) and
\(\EnergySpace\).

\begin{proposition}
\label{prop:analytic-semigroup}
The operator \(-\Aop\) generates a strongly continuous analytic contraction
semigroup
\[
    S(t):=\Semigroup{t},
    \qquad
    t\geq0,
\]
on \(\Hilb\), given spectrally by
\begin{equation}
    S(t)u
    =
    \sum_{j=1}^{\infty}
    e^{-q_jt}u_je_j.
    \label{eq:semigroup-spectral}
\end{equation}
For every \(\alpha\geq0\),
\begin{equation}
    \norm{\Aop^\alpha S(t)}_{\mathcal L(\Hilb)}
    \leq
    C_\alpha t^{-\alpha},
    \qquad
    t>0,
    \label{eq:semigroup-smoothing}
\end{equation}
where \(C_0=1\), while for \(\alpha>0\) one may take
\[
    C_\alpha
    =
    \left(\frac{\alpha}{e}\right)^\alpha.
\]
Moreover,
\[
    S(t)\Dom(\Aop^\theta)
    \subset
    \Dom(\Aop^\theta),
    \qquad
    \norm{S(t)u}_{\Dom(\Aop^\theta)}
    \leq
    \norm{u}_{\Dom(\Aop^\theta)}
\]
for every \(\theta\geq0\) and \(t\geq0\).
\end{proposition}

\begin{proof}
Since \(\Aop\) is positive and self-adjoint, \(-\Aop\) generates an analytic
contraction semigroup on \(\Hilb\); see, for example,
\cite{Pazy1983}.  Its spectral representation is
\eqref{eq:semigroup-spectral}, and strong continuity follows from dominated
convergence.

For \(\alpha>0\), the spectral theorem gives
\[
    \norm{\Aop^\alpha S(t)}_{\mathcal L(\Hilb)}
    =
    \sup_{j\geq1}q_j^\alpha e^{-q_jt}
    \leq
    \sup_{x>0}x^\alpha e^{-xt}
    =
    \left(\frac{\alpha}{e}\right)^\alpha t^{-\alpha}.
\]
For \(\alpha=0\), contractivity gives \(C_0=1\).  Finally, if
\(u\in\Dom(\Aop^\theta)\), then
\[
    \norm{\Aop^\theta S(t)u}_{\Hilb}^2
    =
    \sum_{j=1}^\infty q_j^{2\theta}e^{-2q_jt}|u_j|^2
    \leq
    \norm{\Aop^\theta u}_{\Hilb}^2.
\]
This proves the stated smoothing and fractional-domain estimates.
\end{proof}

\section{The sphere-constrained linear flow and spectral selection}
\label{sec:linear-selection}

We now consider the constrained linear equation.  The normalized semigroup
formula gives the spectral coefficients explicitly, and from that formula one
can read off both support invariance and the selected eigenspace.

\subsection{Normalized semigroup representation}
\label{subsec:normalized-linear-flow}

Recall that
\[
    \Sphere
    =
    \{u\in\Hilb:\norm{u}_{\Hilb}=1\},
\]
and that, for \(u\in\Sphere\),
\[
    \Proj{u}v
    =
    v-\inner{v}{u}_{\Hilb}u
\]
is the orthogonal projection of \(\Hilb\) onto
\[
    \Tan{u}
    =
    \{v\in\Hilb:\inner{v}{u}_{\Hilb}=0\}.
\]
We consider the sphere-constrained linear problem
\begin{equation}
\label{eq:linear-constrained-flow}
    \begin{cases}
    \displaystyle
    \pt u(t)
    =
    -\Proj{u(t)}\Ppoly(\Bop)u(t),
    & t>0,
    \\[1ex]
    u(0)=u_0\in\Sphere.
    \end{cases}
\end{equation}
Whenever \(u(t)\in\Dom(\Ppoly(\Bop))\cap\Sphere\), this equation takes the
equivalent form
\begin{equation}
\label{eq:linear-constrained-expanded}
    \pt u
    =
    -\Ppoly(\Bop)u
    +
    \inner{\Ppoly(\Bop)u}{u}_{\Hilb}u.
\end{equation}

The positive shift introduced in \Cref{sec:operator-setting} provides a
convenient way to define the linear semigroup even when
\(\Ppoly(\Bop)\) is not positive.

\begin{lemma}
\label{lem:polynomial-semigroup}
Let
\[
    S_P(t):=e^{-t\Ppoly(\Bop)},
    \qquad t\geq0.
\]
Then
\begin{equation}
    S_P(t)
    =
    e^{\shift t}e^{-t\Aop},
    \qquad
    t\geq0,
    \label{eq:SP-shift-relation}
\end{equation}
and \(S_P\) is a strongly continuous analytic semigroup on \(\Hilb\).
Moreover,
\begin{equation}
    \norm{S_P(t)}_{\mathcal L(\Hilb)}
    \leq
    e^{-\mu_Pt},
    \qquad
    t\geq0,
    \label{eq:SP-growth-bound}
\end{equation}
where \(\mu_P\) is defined in \eqref{eq:mu-P}.  For every nonzero
\(v\in\Hilb\),
\begin{equation}
    \frac{S_P(t)v}{\norm{S_P(t)v}_{\Hilb}}
    =
    \frac{e^{-t\Aop}v}{\norm{e^{-t\Aop}v}_{\Hilb}},
    \qquad
    t\geq0.
    \label{eq:normalized-shift-independence}
\end{equation}
\end{lemma}

\begin{proof}
Since
\[
    \Aop=\Ppoly(\Bop)+\shift I,
\]
the spectral representations of \(\Aop\) and \(\Ppoly(\Bop)\) with respect
to the eigenbasis \(\{e_j\}_{j\geq1}\) give
\[
    e^{-t\Ppoly(\Bop)}e_j
    =
    e^{-t\Ppoly(\lambda_j)}e_j
\]
and
\[
    e^{-t\Aop}e_j
    =
    e^{-t(\Ppoly(\lambda_j)+\shift)}e_j.
\]
Consequently,
\[
    e^{-t\Ppoly(\Bop)}e_j
    =
    e^{\shift t}e^{-t\Aop}e_j
\]
for every \(j\geq1\), and hence \eqref{eq:SP-shift-relation} follows by
density of the finite spectral sums in \(\Hilb\).

By \Cref{prop:analytic-semigroup}, \(-\Aop\) generates an analytic
contraction semigroup.  Multiplication by the scalar factor
\(e^{\shift t}\) preserves strong continuity and analyticity, so \(S_P\)
is a strongly continuous analytic semigroup.  Since
\[
    \Ppoly(\lambda_j)\geq\mu_P,
\]
we obtain
\[
\begin{aligned}
    \norm{S_P(t)u}_{\Hilb}^2
    &=
    \sum_{j=1}^{\infty}
    e^{-2t\Ppoly(\lambda_j)}|u_j|^2
\\
    &\leq
    e^{-2\mu_Pt}
    \sum_{j=1}^{\infty}|u_j|^2,
\end{aligned}
\]
which proves \eqref{eq:SP-growth-bound}.

Finally, \eqref{eq:SP-shift-relation} gives
\[
    S_P(t)v
    =
    e^{\shift t}e^{-t\Aop}v.
\]
The positive scalar \(e^{\shift t}\) cancels upon normalization, yielding
\eqref{eq:normalized-shift-independence}.
\end{proof}

The constrained equation admits an explicit global solution for every
initial datum on the \(L^2\)-sphere.

\begin{theorem}[Normalized semigroup representation]
\label{thm:normalized-semigroup}
For every \(u_0\in\Sphere\), define
\begin{equation}
\label{eq:normalized-semigroup-formula}
    u(t)
    :=
    \frac{S_P(t)u_0}
         {\norm{S_P(t)u_0}_{\Hilb}}
    =
    \frac{e^{-t\Aop}u_0}
         {\norm{e^{-t\Aop}u_0}_{\Hilb}},
    \qquad
    t\geq0.
\end{equation}
Then
\begin{equation}
    u\in C([0,\infty);\Hilb)
       \cap C((0,\infty);\Dom(\Ppoly(\Bop)))
       \cap C^1((0,\infty);\Hilb),
    \label{eq:linear-solution-regularity}
\end{equation}
\begin{equation}
    u(t)\in\Sphere,
    \qquad
    t\geq0,
    \label{eq:linear-sphere-invariance}
\end{equation}
and \(u\) satisfies \eqref{eq:linear-constrained-flow} for every \(t>0\).
This solution is unique in the class
\[
    C([0,\infty);\Hilb)
    \cap
    C((0,\infty);\Dom(\Ppoly(\Bop)))
    \cap
    C^1((0,\infty);\Hilb).
\]
If \(u_0\in\Dom(\Ppoly(\Bop))\cap\Sphere\), then
\[
    u\in
    C([0,\infty);\Dom(\Ppoly(\Bop)))
    \cap
    C^1([0,\infty);\Hilb),
\]
and the equation holds also at \(t=0\).
\end{theorem}

\begin{proof}
Set
\[
    v(t):=S_P(t)u_0,
    \qquad
    Z(t):=\norm{v(t)}_{\Hilb}.
\]
Since
\[
    v(t)
    =
    \sum_{j=1}^{\infty}
    e^{-t\Ppoly(\lambda_j)}u_{0,j}e_j,
    \qquad
    u_{0,j}:=\inner{u_0}{e_j}_{\Hilb},
\]
and \(u_0\neq0\), one has
\[
    Z(t)^2
    =
    \sum_{j=1}^{\infty}
    e^{-2t\Ppoly(\lambda_j)}|u_{0,j}|^2
    >0
\]
for every \(t\geq0\).  Thus \eqref{eq:normalized-semigroup-formula} is
well defined.

Strong continuity of \(S_P\) gives
\[
    v\in C([0,\infty);\Hilb),
\]
and therefore \(Z\in C([0,\infty))\).  Since \(Z(t)>0\),
\[
    u\in C([0,\infty);\Hilb).
\]
Moreover, the normalization immediately gives
\[
    \norm{u(t)}_{\Hilb}=1,
\]
which proves \eqref{eq:linear-sphere-invariance}.

For \(t>0\), analyticity of \(S_P\) gives
\[
    v(t)\in\Dom(\Ppoly(\Bop))
\]
and
\[
    v_t(t)
    =
    -\Ppoly(\Bop)v(t)
\]
in \(\Hilb\).  Hence \(Z\) is differentiable on \((0,\infty)\), and
\[
\begin{aligned}
    \frac{\dd}{\dd t}Z(t)^2
    &=
    2\inner{v_t(t)}{v(t)}_{\Hilb}
\\
    &=
    -2\inner{\Ppoly(\Bop)v(t)}{v(t)}_{\Hilb}.
\end{aligned}
\]
Since \(v(t)=Z(t)u(t)\), it follows that
\begin{equation}
    Z'(t)
    =
    -Z(t)
    \inner{\Ppoly(\Bop)u(t)}{u(t)}_{\Hilb}.
    \label{eq:Z-prime}
\end{equation}
Differentiating \(u=v/Z\) and using \eqref{eq:Z-prime}, we obtain
\[
\begin{aligned}
    u_t(t)
    &=
    \frac{v_t(t)}{Z(t)}
    -
    \frac{Z'(t)}{Z(t)}u(t)
\\
    &=
    -\Ppoly(\Bop)u(t)
    +
    \inner{\Ppoly(\Bop)u(t)}{u(t)}_{\Hilb}u(t).
\end{aligned}
\]
This is precisely \eqref{eq:linear-constrained-expanded}, hence
\eqref{eq:linear-constrained-flow} holds for every \(t>0\).

To prove uniqueness, let \(\widetilde u\) be another solution in the stated
class.  For \(t>0\), define
\[
    \Lambda(t)
    :=
    \inner{\Ppoly(\Bop)\widetilde u(t)}
           {\widetilde u(t)}_{\Hilb}.
\]
The assumed regularity implies that \(\Lambda\) is continuous on every compact
subinterval of \((0,\infty)\).  Fix \(0<s<t\) and set
\[
    y(r)
    :=
    \exp\left(
        -\int_s^r\Lambda(\tau)\,\dd\tau
    \right)
    \widetilde u(r),
    \qquad
    r\in[s,t].
\]
Using
\[
    \widetilde u_t
    =
    -\Ppoly(\Bop)\widetilde u
    +
    \Lambda\widetilde u,
\]
we obtain
\[
    y_t(r)
    =
    -\Ppoly(\Bop)y(r).
\]
To identify this solution with the semigroup orbit, fix
\(s<r<t\) and define
\[
    z(r):=S_P(t-r)y(r).
\]
For \(r\in(s,t)\), one has \(y(r)\in\Dom(\Ppoly(\Bop))\), and analyticity
of \(S_P\) permits differentiation in \(\Hilb\).  Since
\[
    \frac{\dd}{\dd r}S_P(t-r)v
    =
    S_P(t-r)\Ppoly(\Bop)v,
    \qquad v\in\Dom(\Ppoly(\Bop)),
\]
and \(S_P(t-r)\) commutes with \(\Ppoly(\Bop)\), the product rule gives
\[
\begin{aligned}
    z'(r)
    &=
    S_P(t-r)\Ppoly(\Bop)y(r)
    +
    S_P(t-r)y_t(r)
    =0.
\end{aligned}
\]
Thus \(z\) is constant on \((s,t)\).  Letting \(r\downarrow s\) and using
strong continuity of \(S_P\) and continuity of \(y\), we obtain
\[
    y(t)
    =
    S_P(t-s)y(s)
    =
    S_P(t-s)\widetilde u(s).
\]
and hence
\[
    \widetilde u(t)
    =
    \exp\left(
        \int_s^t\Lambda(\tau)\,\dd\tau
    \right)
    S_P(t-s)\widetilde u(s).
\]
Since \(\norm{\widetilde u(t)}_{\Hilb}=1\), taking the norm of the last
identity determines the scalar factor uniquely and gives
\begin{equation}
    \widetilde u(t)
    =
    \frac{
        S_P(t-s)\widetilde u(s)
    }{
        \norm{S_P(t-s)\widetilde u(s)}_{\Hilb}
    }.
    \label{eq:uniqueness-representation-s}
\end{equation}
Letting \(s\downarrow0\), strong continuity of \(S_P\) and continuity of
\(\widetilde u\) at \(0\) yield
\[
    S_P(t-s)\widetilde u(s)
    \longrightarrow
    S_P(t)u_0
    \quad\text{in }\Hilb.
\]
The denominator in \eqref{eq:uniqueness-representation-s} converges to
\(\norm{S_P(t)u_0}_{\Hilb}>0\).  Consequently,
\[
    \widetilde u(t)
    =
    \frac{S_P(t)u_0}
         {\norm{S_P(t)u_0}_{\Hilb}}
    =
    u(t),
\]
which proves uniqueness.

If \(u_0\in\Dom(\Ppoly(\Bop))\), then the semigroup is strongly continuous
on its generator domain equipped with the graph norm.  Thus
\[
    S_P(\cdot)u_0
    \in
    C([0,\infty);\Dom(\Ppoly(\Bop)))
    \cap
    C^1([0,\infty);\Hilb),
\]
and the same regularity is inherited by the normalized trajectory because its
denominator remains strictly positive.  The differential identity derived
above therefore extends to \(t=0\).
\end{proof}

The representation \eqref{eq:normalized-semigroup-formula} also preserves
every fractional domain of the positive shifted operator.

\begin{corollary}
\label{cor:linear-fractional-regularity}
Let \(\theta\geq0\).  If
\[
    u_0\in\Dom(\Aop^\theta)\cap\Sphere,
\]
then the solution of \eqref{eq:linear-constrained-flow} satisfies
\[
    u\in C([0,\infty);\Dom(\Aop^\theta)).
\]
Moreover, for every \(t>0\),
\begin{equation}
    u(t)\in\bigcap_{\alpha\geq0}\Dom(\Aop^\alpha).
    \label{eq:linear-instantaneous-smoothing}
\end{equation}
\end{corollary}

\begin{proof}
By \Cref{prop:analytic-semigroup},
\[
    e^{-t\Aop}\Dom(\Aop^\theta)
    \subset
    \Dom(\Aop^\theta),
\]
and \(e^{-t\Aop}\) is strongly continuous on \(\Dom(\Aop^\theta)\).
Since
\[
    u(t)
    =
    \frac{e^{-t\Aop}u_0}
         {\norm{e^{-t\Aop}u_0}_{\Hilb}},
\]
and the scalar denominator is positive and continuous, continuity of \(u\)
in \(\Dom(\Aop^\theta)\) follows.

For \(t>0\) and \(\alpha\geq0\),
\[
    \norm{\Aop^\alpha e^{-t\Aop}u_0}_{\Hilb}
    \leq
    C_\alpha t^{-\alpha}\norm{u_0}_{\Hilb}
\]
by \eqref{eq:semigroup-smoothing}.  Hence
\(e^{-t\Aop}u_0\in\Dom(\Aop^\alpha)\) for every \(\alpha\geq0\).
Normalization by a nonzero scalar does not change membership in these
domains, which proves \eqref{eq:linear-instantaneous-smoothing}.
\end{proof}

\subsection{Linear constrained energy}
\label{subsec:linear-energy}

The linear flow is the negative \(L^2\)-gradient flow of the quadratic form
associated with \(\Ppoly(\Bop)\), restricted to the unit sphere.  The shifted
operator provides a convenient definition of this quadratic energy on the
form domain \(\EnergySpace\).

\begin{definition}
\label{def:linear-energy}
For \(u\in\EnergySpace\), define
\begin{equation}
\label{eq:linear-energy}
    \mathcal E_P(u)
    :=
    \frac12
    \norm{\Aop^{1/2}u}_{\Hilb}^2
    -
    \frac{\shift}{2}
    \norm{u}_{\Hilb}^2.
\end{equation}
If \(u\in\Dom(\Ppoly(\Bop))\), then
\begin{equation}
    \mathcal E_P(u)
    =
    \frac12
    \inner{\Ppoly(\Bop)u}{u}_{\Hilb}.
    \label{eq:linear-energy-operator-form}
\end{equation}
\end{definition}

\begin{proposition}[Linear energy identity]
\label{prop:linear-energy-identity}
Let \(u_0\in\EnergySpace\cap\Sphere\), and let \(u\) be the solution given
by \Cref{thm:normalized-semigroup}.  Then, for every \(t\geq0\),
\begin{equation}
\label{eq:linear-energy-identity}
    \mathcal E_P(u(t))
    +
    \int_0^t
    \norm{u_s(s)}_{\Hilb}^2\,\dd s
    =
    \mathcal E_P(u_0).
\end{equation}
In particular,
\[
    t\longmapsto\mathcal E_P(u(t))
\]
is non-increasing.
\end{proposition}

\begin{proof}
Fix \(0<\varepsilon<t\).  By
\Cref{cor:linear-fractional-regularity},
\(u(s)\in\Dom(\Ppoly(\Bop))\) for every \(s\geq\varepsilon\), and the
trajectory is differentiable in \(\Hilb\).  Since
\[
    \norm{u(s)}_{\Hilb}=1,
\]
differentiation gives
\begin{equation}
    \inner{u(s)}{u_s(s)}_{\Hilb}=0.
    \label{eq:linear-tangency}
\end{equation}
Using \eqref{eq:linear-energy-operator-form} and self-adjointness of
\(\Ppoly(\Bop)\),
\[
    \frac{\dd}{\dd s}\mathcal E_P(u(s))
    =
    \inner{\Ppoly(\Bop)u(s)}{u_s(s)}_{\Hilb}.
\]
Equation \eqref{eq:linear-constrained-expanded} can be rewritten as
\[
    \Ppoly(\Bop)u
    =
    -u_t
    +
    \inner{\Ppoly(\Bop)u}{u}_{\Hilb}u.
\]
Consequently, by \eqref{eq:linear-tangency},
\[
\begin{aligned}
    \frac{\dd}{\dd s}\mathcal E_P(u(s))
    &=
    -\norm{u_s(s)}_{\Hilb}^2
    +
    \inner{\Ppoly(\Bop)u(s)}{u(s)}_{\Hilb}
    \inner{u(s)}{u_s(s)}_{\Hilb}
\\
    &=
    -\norm{u_s(s)}_{\Hilb}^2.
\end{aligned}
\]
Integration from \(\varepsilon\) to \(t\) yields
\begin{equation}
    \mathcal E_P(u(t))
    +
    \int_\varepsilon^t
    \norm{u_s(s)}_{\Hilb}^2\,\dd s
    =
    \mathcal E_P(u(\varepsilon)).
    \label{eq:linear-energy-epsilon}
\end{equation}

Because \(u_0\in\EnergySpace\), \Cref{cor:linear-fractional-regularity}
gives
\[
    u(\varepsilon)\longrightarrow u_0
    \quad\text{in }\EnergySpace
    \qquad
    \text{as }\varepsilon\downarrow0.
\]
Hence
\[
    \mathcal E_P(u(\varepsilon))
    \longrightarrow
    \mathcal E_P(u_0).
\]
Because the integrand is nonnegative and the intervals
\((\varepsilon,t)\) increase to \((0,t)\) as
\(\varepsilon\downarrow0\),
\[
    \int_\varepsilon^t\norm{u_s(s)}_{\Hilb}^2\,\dd s
    \uparrow
    \int_0^t\norm{u_s(s)}_{\Hilb}^2\,\dd s.
\]
Together with
\(\mathcal E_P(u(\varepsilon))\to\mathcal E_P(u_0)\), this yields
\eqref{eq:linear-energy-identity}.
\end{proof}

\subsection{Spectral support and stationary states}
\label{subsec:linear-spectral-support}

To describe the asymptotic dynamics, write
\begin{equation}
    p_j:=\Ppoly(\lambda_j),
    \qquad
    c_j:=\inner{u_0}{e_j}_{\Hilb}.
    \label{eq:pj-cj}
\end{equation}
The normalized semigroup formula becomes
\begin{equation}
\label{eq:linear-coefficient-formula}
    u(t)
    =
    \frac{
        \displaystyle
        \sum_{j=1}^{\infty}
        e^{-p_jt}c_je_j
    }{
        \displaystyle
        \left(
            \sum_{j=1}^{\infty}
            e^{-2p_jt}|c_j|^2
        \right)^{1/2}
    }.
\end{equation}
Thus each spectral component evolves only by multiplication by a nonzero
scalar.

\begin{definition}
\label{def:active-spectral-support}
For \(u_0\in\Sphere\), define its active polynomial spectral set by
\begin{equation}
    \Sigma_P(u_0)
    :=
    \left\{
        p_j:
        c_j\neq0
    \right\}.
    \label{eq:active-spectrum}
\end{equation}
Equivalently, if \(\mathsf E_\mu\) denotes the orthogonal projection onto
\[
    \Ker(\Ppoly(\Bop)-\mu I),
\]
then
\begin{equation}
    \Sigma_P(u_0)
    =
    \left\{
        \mu\in\spec(\Ppoly(\Bop)):
        \mathsf E_\mu u_0\neq0
    \right\}.
    \label{eq:active-spectrum-invariant}
\end{equation}
\end{definition}

\begin{proposition}[Conservation of spectral support]
\label{prop:spectral-support-conservation}
Let \(u\) be the solution of \eqref{eq:linear-constrained-flow}.  Then
\begin{equation}
    \inner{u(t)}{e_j}_{\Hilb}
    =
    \frac{
        e^{-p_jt}c_j
    }{
        \left(
            \sum_{k=1}^{\infty}
            e^{-2p_kt}|c_k|^2
        \right)^{1/2}
    },
    \qquad
    j\geq1,\quad t\geq0.
    \label{eq:coefficient-evolution}
\end{equation}
Consequently,
\begin{equation}
    \inner{u(t)}{e_j}_{\Hilb}=0
    \quad\Longleftrightarrow\quad
    c_j=0,
    \qquad
    t\geq0,
    \label{eq:support-conserved-basis}
\end{equation}
and
\begin{equation}
    \Sigma_P(u(t))
    =
    \Sigma_P(u_0),
    \qquad
    t\geq0.
    \label{eq:active-spectrum-conserved}
\end{equation}
More invariantly, for every
\(\mu\in\spec(\Ppoly(\Bop))\),
\begin{equation}
\label{eq:spectral-projection-evolution}
    \mathsf E_\mu u(t)
    =
    \frac{
        e^{-\mu t}
    }{
        \norm{S_P(t)u_0}_{\Hilb}
    }
    \mathsf E_\mu u_0.
\end{equation}
\end{proposition}

\begin{proof}
Formula \eqref{eq:coefficient-evolution} follows directly from
\eqref{eq:linear-coefficient-formula}.  Its denominator is strictly positive,
and \(e^{-p_jt}>0\) for every finite \(t\).  Hence the \(j\)-th coefficient
vanishes at time \(t\) if and only if it vanishes initially, proving
\eqref{eq:support-conserved-basis}.

Since the eigenspaces of the self-adjoint operator \(\Ppoly(\Bop)\) are
orthogonal, applying \(\mathsf E_\mu\) to the normalized semigroup formula
gives
\[
\begin{aligned}
    \mathsf E_\mu u(t)
    &=
    \frac{
        \mathsf E_\mu e^{-t\Ppoly(\Bop)}u_0
    }{
        \norm{S_P(t)u_0}_{\Hilb}
    }
\\
    &=
    \frac{
        e^{-\mu t}\mathsf E_\mu u_0
    }{
        \norm{S_P(t)u_0}_{\Hilb}
    },
\end{aligned}
\]
which is \eqref{eq:spectral-projection-evolution}.  The scalar factor in this
identity is nonzero, so
\[
    \mathsf E_\mu u(t)\neq0
    \quad\Longleftrightarrow\quad
    \mathsf E_\mu u_0\neq0.
\]
This proves \eqref{eq:active-spectrum-conserved}.
\end{proof}

The stationary states of the constrained linear flow are precisely the unit
vectors contained in eigenspaces of \(\Ppoly(\Bop)\).

\begin{proposition}
\label{prop:linear-equilibria}
Let \(\phi\in\Dom(\Ppoly(\Bop))\cap\Sphere\).  Then the following statements
are equivalent:
\begin{equation}
    \Proj{\phi}\Ppoly(\Bop)\phi=0,
    \label{eq:linear-equilibrium-projected}
\end{equation}
and
\begin{equation}
    \Ppoly(\Bop)\phi=\mu\phi
    \quad\text{for some }\mu\in\RR.
    \label{eq:linear-equilibrium-eigenvalue}
\end{equation}
In this case,
\[
    \mu
    =
    \inner{\Ppoly(\Bop)\phi}{\phi}_{\Hilb}.
\]
Consequently, the equilibrium set is
\begin{equation}
\label{eq:linear-equilibrium-set}
    \EqSetLin
    =
    \bigcup_{\mu\in\spec(\Ppoly(\Bop))}
    \left(
        \Ker(\Ppoly(\Bop)-\mu I)
        \cap\Sphere
    \right).
\end{equation}
\end{proposition}

\begin{proof}
If \eqref{eq:linear-equilibrium-projected} holds, then by the definition of
the tangent projection,
\[
    \Ppoly(\Bop)\phi
    -
    \inner{\Ppoly(\Bop)\phi}{\phi}_{\Hilb}\phi
    =
    0.
\]
Thus \eqref{eq:linear-equilibrium-eigenvalue} holds with
\[
    \mu
    =
    \inner{\Ppoly(\Bop)\phi}{\phi}_{\Hilb}.
\]
Conversely, if
\[
    \Ppoly(\Bop)\phi=\mu\phi,
\]
then, since \(\norm{\phi}_{\Hilb}=1\),
\[
\begin{aligned}
    \Proj{\phi}\Ppoly(\Bop)\phi
    &=
    \mu\phi
    -
    \inner{\mu\phi}{\phi}_{\Hilb}\phi
\\
    &=
    \mu\phi-\mu\phi
    =
    0.
\end{aligned}
\]
This proves the equivalence and the characterization
\eqref{eq:linear-equilibrium-set}.
\end{proof}

\subsection{Spectral selection}
\label{subsec:spectral-selection-theorem}

Since
\[
    p_j=\Ppoly(\lambda_j)\longrightarrow+\infty,
\]
the active spectral set \(\Sigma_P(u_0)\) possesses a smallest element.
This value determines the asymptotic state.

\begin{definition}
\label{def:selected-spectral-value}
For \(u_0\in\Sphere\), define
\begin{equation}
\label{eq:mu-star}
    \mu_*(u_0)
    :=
    \min\Sigma_P(u_0).
\end{equation}
Let
\[
    \mathsf E_*
    :=
    \mathsf E_{\mu_*(u_0)}
\]
be the orthogonal projection onto
\begin{equation}
    \mathcal H_*
    :=
    \Ker\bigl(
        \Ppoly(\Bop)-\mu_*(u_0)I
    \bigr),
    \label{eq:selected-eigenspace}
\end{equation}
and set
\begin{equation}
    w_*:=\mathsf E_*u_0.
    \label{eq:w-star}
\end{equation}
Then \(w_*\neq0\).
\end{definition}

\begin{remark}
\label{rem:selected-eigenspace-finite}
The eigenspace \(\mathcal H_*\) is finite-dimensional.  By
\Cref{lem:dirichlet-spectral}, the Dirichlet spectrum has no finite
accumulation point, and
\[
    \Ppoly(\lambda_j)\longrightarrow+\infty.
\]
Hence only finitely many indices \(j\) satisfy
\(\Ppoly(\lambda_j)=\mu_*(u_0)\).  Equivalently, the equation
\[
    \Ppoly(\lambda)=\mu_*(u_0)
\]
has only finitely many real roots, and each corresponding Dirichlet
eigenspace has finite multiplicity.
\end{remark}

\begin{theorem}[Spectral selection]
\label{thm:linear-spectral-selection}
Let \(u_0\in\Sphere\), and let \(u\) be the solution of
\eqref{eq:linear-constrained-flow}.  Then
\begin{equation}
\label{eq:linear-limit}
    u(t)
    \longrightarrow
    u_\infty
    :=
    \frac{\mathsf E_*u_0}
         {\norm{\mathsf E_*u_0}_{\Hilb}}
    \quad\text{strongly in }\Hilb
    \qquad
    \text{as }t\to\infty.
\end{equation}
The limit \(u_\infty\) belongs to
\[
    \mathcal H_*\cap\Sphere
\]
and is therefore an equilibrium of the constrained linear flow.

If
\[
    \Sigma_P(u_0)\setminus\{\mu_*(u_0)\}\neq\varnothing,
\]
define the active spectral gap
\begin{equation}
\label{eq:active-spectral-gap}
    \delta_*(u_0)
    :=
    \min
    \left\{
        \mu-\mu_*(u_0):
        \mu\in\Sigma_P(u_0),\
        \mu>\mu_*(u_0)
    \right\}.
\end{equation}
Then
\[
    \delta_*(u_0)>0
\]
and there exists a constant \(C(u_0)>0\) such that
\begin{equation}
\label{eq:H-spectral-selection-rate}
    \norm{u(t)-u_\infty}_{\Hilb}
    \leq
    C(u_0)e^{-\delta_*(u_0)t},
    \qquad
    t\geq0.
\end{equation}
If
\[
    \Sigma_P(u_0)=\{\mu_*(u_0)\},
\]
then \(u(t)=u_0\) for every \(t\geq0\).

More generally, if
\[
    u_0\in\Dom(\Aop^\theta)\cap\Sphere
\]
for some \(\theta\geq0\), then
\begin{equation}
\label{eq:fractional-spectral-convergence}
    u(t)\longrightarrow u_\infty
    \quad\text{in }\Dom(\Aop^\theta).
\end{equation}
If the active spectral gap is defined, then for each fixed \(\theta\geq0\)
there exists a constant \(C_{\theta}(u_0)>0\), depending on \(\theta\) and
\(u_0\), such that
\begin{equation}
\label{eq:fractional-spectral-rate}
    \norm{
        \Aop^\theta(u(t)-u_\infty)
    }_{\Hilb}
    \leq
    C_{\theta}(u_0)
    e^{-\delta_*(u_0)t},
    \qquad
    t\geq0.
\end{equation}
More explicitly, with
\[
    w_*:=\mathsf E_*u_0,
    \qquad
    w_\perp:=(I-\mathsf E_*)u_0,
    \qquad
    a:=\norm{w_*}_{\Hilb},
\]
one may take
\begin{equation}
\label{eq:explicit-fractional-rate-constant}
    C_\theta(u_0)
    =
    \frac{\norm{\Aop^\theta w_\perp}_{\Hilb}}{a}
    +
    \frac{
        \norm{\Aop^\theta w_*}_{\Hilb}
        \norm{w_\perp}_{\Hilb}^2
    }{2a^3}.
\end{equation}
\end{theorem}

\begin{proof}
Set
\[
    \mu_*:=\mu_*(u_0),
    \qquad
    w_*:=\mathsf E_*u_0,
    \qquad
    a:=\norm{w_*}_{\Hilb}>0.
\]
Decompose
\[
    u_0=w_*+w_\perp,
    \qquad
    w_\perp\perp\mathcal H_*.
\]
Using the spectral representation of \(S_P(t)\), we may factor out
\(e^{-\mu_*t}\) and write
\begin{equation}
    S_P(t)u_0
    =
    e^{-\mu_*t}
    \bigl(
        w_*+r(t)
    \bigr),
    \label{eq:SP-decomposition}
\end{equation}
where
\begin{equation}
    r(t)
    :=
    \sum_{\substack{j\geq1\\ p_j>\mu_*}}
    e^{-t(p_j-\mu_*)}c_je_j.
    \label{eq:r-definition}
\end{equation}
Terms with \(p_j<\mu_*\) do not occur because the definition of \(\mu_*\)
implies \(c_j=0\) for such indices, while terms satisfying
\(p_j=\mu_*\) are precisely the components of \(w_*\).  Hence
\begin{equation}
    r(t)\perp w_*.
    \label{eq:r-orthogonal}
\end{equation}

Suppose first that
\[
    \Sigma_P(u_0)\setminus\{\mu_*\}\neq\varnothing.
\]
Since the values \(p_j\) tend to \(+\infty\), the set of active values
strictly larger than \(\mu_*\) has a smallest member.  Therefore
\(\delta_*:=\delta_*(u_0)>0\).  By Parseval's identity,
\[
\begin{aligned}
    \norm{r(t)}_{\Hilb}^2
    &=
    \sum_{\substack{j\geq1\\p_j>\mu_*}}
    e^{-2t(p_j-\mu_*)}|c_j|^2
\\
    &\leq
    e^{-2\delta_*t}
    \sum_{\substack{j\geq1\\p_j>\mu_*}}
    |c_j|^2
\\
    &=
    e^{-2\delta_*t}
    \norm{w_\perp}_{\Hilb}^2.
\end{aligned}
\]
Thus
\begin{equation}
\label{eq:r-H-rate}
    \norm{r(t)}_{\Hilb}
    \leq
    e^{-\delta_*t}
    \norm{w_\perp}_{\Hilb}.
\end{equation}

By \eqref{eq:r-orthogonal},
\[
    \norm{w_*+r(t)}_{\Hilb}^2
    =
    a^2+\norm{r(t)}_{\Hilb}^2.
\]
The normalized representation therefore gives
\begin{equation}
\label{eq:u-w-r}
    u(t)
    =
    \frac{
        w_*+r(t)
    }{
        \bigl(
            a^2+\norm{r(t)}_{\Hilb}^2
        \bigr)^{1/2}
    }.
\end{equation}
Since \(r(t)\to0\) in \(\Hilb\), it follows immediately that
\[
    u(t)\longrightarrow\frac{w_*}{a}=u_\infty
    \quad\text{in }\Hilb.
\]

For the quantitative estimate, set
\[
    D(t)
    :=
    \bigl(
        a^2+\norm{r(t)}_{\Hilb}^2
    \bigr)^{1/2}.
\]
Then \(D(t)\geq a\), and
\[
\begin{aligned}
    \norm{u(t)-u_\infty}_{\Hilb}
    &\leq
    \norm{
        \frac{r(t)}{D(t)}
    }_{\Hilb}
    +
    \norm{
        \left(
            \frac1{D(t)}-\frac1a
        \right)w_*
    }_{\Hilb}
\\
    &\leq
    \frac{\norm{r(t)}_{\Hilb}}{a}
    +
    \frac{D(t)-a}{D(t)}.
\end{aligned}
\]
Since
\[
    D(t)-a
    =
    \frac{
        \norm{r(t)}_{\Hilb}^2
    }{
        D(t)+a
    },
\]
we obtain
\[
    \frac{D(t)-a}{D(t)}
    \leq
    \frac{
        \norm{r(t)}_{\Hilb}^2
    }{
        2a^2
    }.
\]
Using \eqref{eq:r-H-rate},
\[
\begin{aligned}
    \norm{u(t)-u_\infty}_{\Hilb}
    &\leq
    \frac{
        \norm{w_\perp}_{\Hilb}
    }{a}
    e^{-\delta_*t}
\\
    &\quad+
    \frac{
        \norm{w_\perp}_{\Hilb}^2
    }{
        2a^2
    }
    e^{-2\delta_*t}.
\end{aligned}
\]
Since \(e^{-2\delta_*t}\leq e^{-\delta_*t}\), this proves
\eqref{eq:H-spectral-selection-rate}.

If
\[
    \Sigma_P(u_0)=\{\mu_*\},
\]
then \(u_0\in\mathcal H_*\) and
\[
    S_P(t)u_0=e^{-\mu_*t}u_0.
\]
Normalization gives
\[
    u(t)=u_0,
\]
so the trajectory is stationary.

Assume now that
\[
    u_0\in\Dom(\Aop^\theta)
\]
for some \(\theta\geq0\).  Since the spectral projection
\(\mathsf E_*\) commutes with \(\Aop\),
\[
    w_*,w_\perp\in\Dom(\Aop^\theta).
\]
Furthermore,
\[
\begin{aligned}
    \norm{\Aop^\theta r(t)}_{\Hilb}^2
    &=
    \sum_{\substack{j\geq1\\p_j>\mu_*}}
    q_j^{2\theta}
    e^{-2t(p_j-\mu_*)}
    |c_j|^2
\\
    &\leq
    e^{-2\delta_*t}
    \sum_{\substack{j\geq1\\p_j>\mu_*}}
    q_j^{2\theta}|c_j|^2
\\
    &\leq
    e^{-2\delta_*t}
    \norm{\Aop^\theta w_\perp}_{\Hilb}^2.
\end{aligned}
\]
Thus
\begin{equation}
\label{eq:r-fractional-rate}
    \norm{\Aop^\theta r(t)}_{\Hilb}
    \leq
    e^{-\delta_*t}
    \norm{\Aop^\theta w_\perp}_{\Hilb}.
\end{equation}
Since
\[
    u(t)-u_\infty
    =
    \frac{r(t)}{D(t)}
    +
    \left(
        \frac1{D(t)}-\frac1a
    \right)w_*,
\]
application of \(\Aop^\theta\) gives
\[
\begin{aligned}
    \norm{
        \Aop^\theta(u(t)-u_\infty)
    }_{\Hilb}
    &\leq
    \frac{
        \norm{\Aop^\theta r(t)}_{\Hilb}
    }{a}
\\
    &\quad+
    \left|
        \frac1{D(t)}-\frac1a
    \right|
    \norm{\Aop^\theta w_*}_{\Hilb}.
\end{aligned}
\]
By \eqref{eq:r-fractional-rate}, the first term is bounded by
\[
    \frac{\norm{\Aop^\theta w_\perp}_{\Hilb}}{a}
    e^{-\delta_*t}.
\]
Moreover,
\[
    \left|\frac1{D(t)}-\frac1a\right|
    =
    \frac{\norm{r(t)}_{\Hilb}^2}{aD(t)(D(t)+a)}
    \leq
    \frac{\norm{r(t)}_{\Hilb}^2}{2a^3},
\]
and \eqref{eq:r-H-rate} therefore gives
\[
    \left|\frac1{D(t)}-\frac1a\right|
    \norm{\Aop^\theta w_*}_{\Hilb}
    \leq
    \frac{
        \norm{\Aop^\theta w_*}_{\Hilb}
        \norm{w_\perp}_{\Hilb}^2
    }{2a^3}
    e^{-2\delta_*t}.
\]
Since \(e^{-2\delta_*t}\leq e^{-\delta_*t}\), this proves
\eqref{eq:fractional-spectral-rate} with the constant in
\eqref{eq:explicit-fractional-rate-constant}, and hence
\eqref{eq:fractional-spectral-convergence}.
\end{proof}

\begin{corollary}[Positive-time fractional convergence for rough data]
\label{cor:rough-data-fractional-rate}
Let \(u_0\in\Sphere\), let \(\theta\geq0\), and let \(\tau>0\).  Then
\[
    u(t)\in\Dom(\Aop^\theta),
    \qquad t>0,
\]
and
\begin{equation}
\label{eq:rough-data-fractional-limit}
    u(t)\longrightarrow u_\infty
    \quad\text{in }\Dom(\Aop^\theta)
    \qquad\text{as }t\to\infty.
\end{equation}
If the active spectral gap \(\delta_*(u_0)\) is defined, then there exists
\(C_{\theta,\tau}(u_0)>0\) such that
\begin{equation}
\label{eq:rough-data-fractional-rate}
    \norm{\Aop^\theta(u(t)-u_\infty)}_{\Hilb}
    \leq
    C_{\theta,\tau}(u_0)e^{-\delta_*(u_0)t},
    \qquad t\geq\tau.
\end{equation}
Thus no fractional regularity of the initial datum is required for the
positive-time rate.  For every fixed \(\tau>0\), the constant
\(C_{\theta,\tau}(u_0)\) is finite; no bound uniform as
\(\tau\downarrow0\) is asserted.
\end{corollary}

\begin{proof}
Instantaneous membership in \(\Dom(\Aop^\theta)\) follows from the analytic
semigroup representation in \Cref{thm:normalized-semigroup}.  It remains to
prove the rate.  Use the notation in the proof of
\Cref{thm:linear-spectral-selection} and set
\[
    \delta_*:=\delta_*(u_0),
    \qquad
    \gamma_j:=p_j-\mu_*>0
\]
for the active indices outside the selected space.  The set
\[
    J_*:=\{j:c_j\neq0,\ \gamma_j=\delta_*\}
\]
is finite because \(p_j\to\infty\).  Split
\[
    r(t)=r_*(t)+r_>(t),
\]
where \(r_*\) contains the indices in \(J_*\) and \(r_>\) the remaining
active indices with \(\gamma_j>\delta_*\).  Since \(J_*\) is finite,
\[
    \norm{\Aop^\theta r_*(t)}_{\Hilb}
    \leq
    C_*(u_0,\theta)e^{-\delta_*t}.
\]
For \(t\geq\tau\),
\[
\begin{aligned}
    \norm{\Aop^\theta r_>(t)}_{\Hilb}^2
    &=
    e^{-2\delta_*t}
    \sum_{\substack{j:\ c_j\neq0\\ \gamma_j>\delta_*}}
    q_j^{2\theta}
    e^{-2t(\gamma_j-\delta_*)}|c_j|^2
\\
    &\leq
    e^{-2\delta_*t}
    M_{\theta,\tau}^2
    \sum_{j=1}^\infty |c_j|^2,
\end{aligned}
\]
where
\[
    M_{\theta,\tau}
    :=
    \sup_{\substack{j:\ c_j\neq0\\ \gamma_j>\delta_*}}
    q_j^\theta e^{-\tau(\gamma_j-\delta_*)}.
\]
The supremum is finite.  Indeed,
\(q_j=p_j+\shift\) and \(\gamma_j=p_j-\mu_*\), so
\(\gamma_j-\delta_*\to\infty\) and
\[
    q_j^\theta e^{-\tau(\gamma_j-\delta_*)}\longrightarrow0
    \qquad\text{as }j\to\infty.
\]
Thus
\[
    \norm{\Aop^\theta r(t)}_{\Hilb}
    \leq
    C_{\theta,\tau}'(u_0)e^{-\delta_*t},
    \qquad t\geq\tau.
\]
The selected vector \(w_*\) belongs to every fractional domain because
\(\mathcal H_*\) is finite-dimensional.  The decomposition
\[
    u(t)-u_\infty
    =
    \frac{r(t)}{D(t)}
    +
    \left(\frac1{D(t)}-\frac1a\right)w_*
\]
from the proof of \Cref{thm:linear-spectral-selection} uses the same scalar
\(D(t)\), independently of \(\theta\).  Since \(D(t)\geq a>0\) and
\[
    \left|\frac1{D(t)}-\frac1a\right|
    \leq
    \frac{\norm{r(t)}_{\Hilb}^2}{2a^3},
\]
the second term decays at least like \(e^{-2\delta_*t}\) in every
fractional norm.  This proves \eqref{eq:rough-data-fractional-rate} and
\eqref{eq:rough-data-fractional-limit}.
\end{proof}

The active gap depends on both the polynomial and the spectral support of the
initial datum.  It is positive for each fixed nonstationary datum, but it need
not stay uniformly positive when the datum or the polynomial varies.

We record separately the simple and degenerate cases.

\begin{corollary}[Simple selected mode]
\label{cor:simple-selected-mode}
Assume that
\[
    \dim\mathcal H_*=1.
\]
Let \(e_*\) be a unit vector spanning \(\mathcal H_*\).  Then
\[
    \inner{u_0}{e_*}_{\Hilb}\neq0
\]
and
\begin{equation}
\label{eq:simple-selected-limit}
    u(t)
    \longrightarrow
    \operatorname{sgn}
    \bigl(
        \inner{u_0}{e_*}_{\Hilb}
    \bigr)e_*
    \quad\text{in }\Hilb.
\end{equation}
If \(u_0\in\Dom(\Aop^\theta)\), the same convergence holds in
\(\Dom(\Aop^\theta)\).
\end{corollary}

\begin{proof}
Since \(\mathcal H_*=\Span\{e_*\}\),
\[
    \mathsf E_*u_0
    =
    \inner{u_0}{e_*}_{\Hilb}e_*.
\]
The definition of \(\mu_*(u_0)\) implies that this projection is nonzero.
Therefore
\[
\begin{aligned}
    \frac{\mathsf E_*u_0}
         {\norm{\mathsf E_*u_0}_{\Hilb}}
    &=
    \frac{
        \inner{u_0}{e_*}_{\Hilb}
    }{
        |\inner{u_0}{e_*}_{\Hilb}|
    }e_*
\\
    &=
    \operatorname{sgn}
    \bigl(
        \inner{u_0}{e_*}_{\Hilb}
    \bigr)e_*.
\end{aligned}
\]
The conclusion follows from
\Cref{thm:linear-spectral-selection}.
\end{proof}

\begin{corollary}[Degenerate selected eigenspace]
\label{cor:degenerate-selected-space}
Assume that
\[
    \dim\mathcal H_*=r\geq2.
\]
Then the asymptotic state is
\[
    u_\infty
    =
    \frac{\mathsf E_*u_0}
         {\norm{\mathsf E_*u_0}_{\Hilb}},
\]
and the direction of the projection \(\mathsf E_*u_0\) within
\(\mathcal H_*\) is preserved by the flow.  In particular, the linear
constrained dynamics do not distinguish between different unit vectors in
the same eigenspace \(\mathcal H_*\).
\end{corollary}

\begin{proof}
By \eqref{eq:spectral-projection-evolution},
\[
    \mathsf E_*u(t)
    =
    \frac{
        e^{-\mu_*t}
    }{
        \norm{S_P(t)u_0}_{\Hilb}
    }
    \mathsf E_*u_0.
\]
Thus \(\mathsf E_*u(t)\) is a positive scalar multiple of
\(\mathsf E_*u_0\) for every \(t\geq0\), so its direction in
\(\mathcal H_*\) is independent of time.  The asserted limit follows from
\Cref{thm:linear-spectral-selection}.
\end{proof}

The preceding result depends on the spectral support of the initial datum,
rather than only on the global minimum of the polynomial over the Dirichlet
spectrum.  The following consequence isolates the case in which the globally
preferred eigenspace is present initially.

\begin{corollary}[Selection of a global minimizing eigenspace]
\label{cor:global-minimizer-selection}
Let
\[
    \mu_{\min}
    :=
    \min_{j\geq1}\Ppoly(\lambda_j)
\]
and let
\[
    \mathcal H_{\min}
    :=
    \Ker\bigl(
        \Ppoly(\Bop)-\mu_{\min}I
    \bigr).
\]
Denote by \(\mathsf E_{\min}\) the orthogonal projection onto
\(\mathcal H_{\min}\).  If
\[
    \mathsf E_{\min}u_0\neq0,
\]
then
\[
    \mu_*(u_0)=\mu_{\min}
\]
and
\begin{equation}
\label{eq:global-minimizer-limit}
    u(t)
    \longrightarrow
    \frac{
        \mathsf E_{\min}u_0
    }{
        \norm{\mathsf E_{\min}u_0}_{\Hilb}
    }
    \quad\text{in }\Hilb.
\end{equation}
If
\[
    \mathsf E_{\min}u_0=0,
\]
then the flow remains orthogonal to \(\mathcal H_{\min}\) for all
\(t\geq0\) and selects the smallest polynomial spectral value contained in
the initial spectral support.
\end{corollary}

\begin{proof}
If \(\mathsf E_{\min}u_0\neq0\), then
\(\mu_{\min}\in\Sigma_P(u_0)\).  Since \(\mu_{\min}\) is the smallest
spectral value of \(\Ppoly(\Bop)\), it follows that
\[
    \mu_*(u_0)=\mu_{\min}.
\]
The convergence statement is therefore a direct consequence of
\Cref{thm:linear-spectral-selection}.

If \(\mathsf E_{\min}u_0=0\), then
\eqref{eq:spectral-projection-evolution} gives
\[
    \mathsf E_{\min}u(t)=0
\]
for every \(t\geq0\).  Thus the flow cannot generate a component in the
globally minimizing eigenspace.  The selected value is consequently the
smallest element of the remaining active spectrum, as stated.
\end{proof}

\begin{corollary}
\label{cor:linear-total-dissipation}
Let \(u_0\in\EnergySpace\cap\Sphere\), and let
\(\mu_*=\mu_*(u_0)\).  Then
\begin{equation}
    \lim_{t\to\infty}\mathcal E_P(u(t))
    =
    \frac{\mu_*}{2},
    \label{eq:linear-energy-limit}
\end{equation}
and
\begin{equation}
\label{eq:linear-total-dissipation}
    \int_0^\infty
    \norm{u_t(t)}_{\Hilb}^2\,\dd t
    =
    \mathcal E_P(u_0)
    -
    \frac{\mu_*}{2}.
\end{equation}
\end{corollary}

\begin{proof}
By \Cref{thm:linear-spectral-selection},
\[
    u(t)\longrightarrow u_\infty
    \quad\text{in }\EnergySpace,
\]
because \(u_0\in\EnergySpace\).  Since
\[
    u_\infty\in
    \Ker(\Ppoly(\Bop)-\mu_*I)\cap\Sphere,
\]
we have
\[
    \mathcal E_P(u_\infty)
    =
    \frac12
    \inner{\Ppoly(\Bop)u_\infty}{u_\infty}_{\Hilb}
    =
    \frac{\mu_*}{2}.
\]
Continuity of \(\mathcal E_P\) on \(\EnergySpace\) therefore yields
\eqref{eq:linear-energy-limit}.  Letting \(t\to\infty\) in the energy
identity \eqref{eq:linear-energy-identity} and using monotone convergence
for the nonnegative dissipation integral gives
\eqref{eq:linear-total-dissipation}.
\end{proof}

\begin{remark}
\label{rem:polynomial-selection-mechanism}
The selected modes minimize the values
\[
    \Ppoly(\lambda_j)
\]
rather than the Dirichlet eigenvalues \(\lambda_j\) themselves.  Thus the
ordering induced by the constrained flow need not agree with the ordering
of the spectrum of \(\Bop\).  For a monotone increasing polynomial, the
mechanism reduces to the familiar preference for the lowest active
Dirichlet mode.  For a non-monotone polynomial, interior spectral bands may
be preferred.  The Swift--Hohenberg-type choice
\[
    \Ppoly(s)=(s-\rhoSH)^2
\]
will be examined in \Cref{sec:examples}; in that case the selected modes are
those active Dirichlet modes for which \(|\lambda_j-\rhoSH|\) is smallest.
\end{remark}

\begin{example}[Interior and degenerate spectral selection]
\label{ex:finite-spectral-selection}
Let \(\Om=(0,\pi)\).  The Dirichlet eigenvalues are
\(\lambda_j=j^2\), so \(\lambda_2=4\) and \(\lambda_3=9\).  Suppose the
initial datum has nonzero components in both corresponding eigenspaces.  For
\[
    \Ppoly(s)=(s-4)^2,
\]
the level \(\lambda=4\) is selected whenever no other active eigenvalue has
the same polynomial value.  If instead
\[
    \Ppoly(s)=\left(s-\frac{13}{2}\right)^2,
\]
then
\[
    \Ppoly(4)=\Ppoly(9)=\frac{25}{4}.
\]
If all remaining active Dirichlet eigenvalues lie farther from \(13/2\), the
selected eigenspace is the orthogonal direct sum of the eigenspaces for
\(4\) and \(9\).  Thus a non-monotone polynomial can select an interior
spectral level or several distinct Dirichlet levels simultaneously.
\end{example}

\begin{remark}
\label{rem:shift-selection}
The positive shift \(\shift\) has no influence on spectral selection.  Indeed,
the eigenvalues of \(\Aop\) are
\[
    q_j=\Ppoly(\lambda_j)+\shift,
\]
so
\[
    q_j-q_k
    =
    \Ppoly(\lambda_j)-\Ppoly(\lambda_k).
\]
Hence the minimizing eigenspaces, the active spectral gap, the normalized
trajectory, and its asymptotic state are identical whether the flow is
written in terms of \(\Ppoly(\Bop)\) or the positive operator \(\Aop\).
\end{remark}

\section{Spectral design and stability under polynomial perturbations}
\label{sec:design-robustness}

The polynomial may also be chosen with a prescribed minimizing set.  We first
give an explicit construction and then ask whether the selected set survives
a small change of the polynomial.

For a Dirichlet eigenvalue \(\lambda\), denote by
\(\mathsf E^{\Bop}_{\lambda}\) the orthogonal projection onto
\(\Ker(\Bop-\lambda I)\).  For \(u_0\in\Sphere\), define its active
Dirichlet spectrum by
\begin{equation}
\label{eq:active-dirichlet-spectrum-general}
    \Sigma_{\Bop}(u_0)
    :=
    \left\{
        \lambda\in\spec(\Bop):
        \mathsf E^{\Bop}_{\lambda}u_0\neq0
    \right\}.
\end{equation}

\begin{theorem}[Prescribed finite spectral selection]
\label{thm:prescribed-spectral-selection}
Let
\[
    J=\{\nu_1,\ldots,\nu_r\}
    \subset\spec(\Bop)
\]
be a finite set of distinct Dirichlet eigenvalues and define
\begin{equation}
\label{eq:prescribed-polynomial}
    \Ppoly_J(s)
    :=
    \prod_{k=1}^r(s-\nu_k)^2.
\end{equation}
Then \(\Ppoly_J\) is a real polynomial with positive leading coefficient,
\[
    \min_{\lambda\in\spec(\Bop)}\Ppoly_J(\lambda)=0,
\]
and its globally minimizing spectral subspace is exactly
\begin{equation}
\label{eq:prescribed-minimizing-space}
    \mathcal H_J
    :=
    \bigoplus_{k=1}^r\Ker(\Bop-\nu_kI).
\end{equation}
If \(\mathsf E_J\) denotes the orthogonal projection onto \(\mathcal H_J\)
and \(\mathsf E_Ju_0\neq0\), then the constrained flow generated by
\(\Ppoly_J(\Bop)\) satisfies
\begin{equation}
\label{eq:prescribed-selection-limit}
    u(t)
    \longrightarrow
    \frac{\mathsf E_Ju_0}{\norm{\mathsf E_Ju_0}_{\Hilb}}
    \qquad\text{in }\Hilb.
\end{equation}
If the initial datum has an active component outside \(\mathcal H_J\), the
convergence is exponential.
\end{theorem}

\begin{proof}
Each factor in \eqref{eq:prescribed-polynomial} is nonnegative on \(\RR\),
so \(\Ppoly_J(s)\geq0\).  Its leading term is \(s^{2r}\), hence the leading
coefficient is positive.  Moreover,
\[
    \Ppoly_J(\lambda)=0
    \quad\Longleftrightarrow\quad
    \lambda\in J
\]
for every \(\lambda\in\RR\).  Since each \(\nu_k\) belongs to the
Dirichlet spectrum, the minimum of \(\Ppoly_J\) on
\(\spec(\Bop)\) is zero and the associated eigenspace of
\(\Ppoly_J(\Bop)\) is precisely \(\mathcal H_J\).  If
\(\mathsf E_Ju_0\neq0\), zero belongs to the active polynomial spectrum and
is its smallest value.  The limit \eqref{eq:prescribed-selection-limit}
therefore follows from \Cref{thm:linear-spectral-selection}.  If an active
component lies outside \(\mathcal H_J\), then the remaining active
polynomial values are positive.  Since they form a discrete set tending to
\(+\infty\), the active gap above zero is positive, and the exponential
estimate follows from \eqref{eq:H-spectral-selection-rate}.
\end{proof}

The construction has degree \(2r\).  This is the least possible degree among
real polynomials that are nonnegative on \(\RR\) and vanish at the prescribed
\(r\) points, since each such zero has even multiplicity.  No minimality is
claimed if one imposes the minimizing condition only on the discrete
Dirichlet spectrum.

We next consider perturbations of a fixed polynomial.

\begin{theorem}[Stability under polynomial perturbations]
\label{thm:polynomial-perturbation-stability}
Let \(\Ppoly\) have degree \(\degree\) and positive leading coefficient,
let \(u_0\in\Sphere\), and set
\[
    \mu_*:=\mu_*(u_0),
    \qquad
    S_*:=
    \left\{
        \lambda\in\Sigma_{\Bop}(u_0):
        \Ppoly(\lambda)=\mu_*
    \right\}.
\]
Let \(R\) be a real polynomial with \(\deg R\leq\degree\), and for
\(\varepsilon\in\RR\) set
\[
    \Ppoly_\varepsilon:=\Ppoly+\varepsilon R.
\]
Define
\begin{equation}
\label{eq:perturbation-ratio}
    M_R(u_0)
    :=
    \sup_{\lambda\in\Sigma_{\Bop}(u_0)\setminus S_*}
    \max_{\nu\in S_*}
    \frac{|R(\lambda)-R(\nu)|}
         {\Ppoly(\lambda)-\mu_*},
\end{equation}
with the convention \(M_R(u_0)=0\) when
\(\Sigma_{\Bop}(u_0)=S_*\).  Then \(M_R(u_0)<\infty\).

Suppose \(|\varepsilon|M_R(u_0)<1\) and the leading coefficient of
\(\Ppoly_\varepsilon\) is positive.  Then no active Dirichlet level outside
\(S_*\) minimizes \(\Ppoly_\varepsilon\).  More precisely, the active
minimizing set for \(\Ppoly_\varepsilon\) is
\begin{equation}
\label{eq:perturbed-selected-set}
    S_{*,\varepsilon}
    =
    \operatorname*{argmin}_{\nu\in S_*}
    \bigl(\varepsilon R(\nu)\bigr).
\end{equation}
Consequently, if \(S_*=\{\lambda_*\}\) is a single Dirichlet level, then
\(\lambda_*\) remains the unique selected active level for all such
\(\varepsilon\).  If \(S_*\) contains several levels, a small perturbation
may split the degeneracy, and the surviving selected levels are exactly those
specified by \eqref{eq:perturbed-selected-set}.
\end{theorem}

\begin{proof}
The set \(S_*\) is finite by \Cref{rem:selected-eigenspace-finite}.  For
\(\lambda\in\Sigma_{\Bop}(u_0)\setminus S_*\),
\[
    \Ppoly(\lambda)-\mu_*>0.
\]
Because \(\deg R\leq\degree\) and
\(\Ppoly(\lambda)\sim a_{\degree}\lambda^\degree\) along the Dirichlet
spectrum, the quotient in \eqref{eq:perturbation-ratio} is bounded for large
\(\lambda\); only finitely many remaining active levels lie in a bounded
spectral interval.  Hence \(M_R(u_0)<\infty\).

Let \(\lambda\in\Sigma_{\Bop}(u_0)\setminus S_*\) and \(\nu\in S_*\).
Then
\[
\begin{aligned}
    \Ppoly_\varepsilon(\lambda)
    -\Ppoly_\varepsilon(\nu)
    &=
    \Ppoly(\lambda)-\mu_*
    +
    \varepsilon\bigl(R(\lambda)-R(\nu)\bigr)
\\
    &\geq
    \bigl(1-|\varepsilon|M_R(u_0)\bigr)
    \bigl(\Ppoly(\lambda)-\mu_*\bigr)
    >0.
\end{aligned}
\]
Thus every active level outside \(S_*\) has strictly larger perturbed
polynomial value than every level in \(S_*\).  On \(S_*\),
\[
    \Ppoly_\varepsilon(\nu)
    =
    \mu_*+\varepsilon R(\nu),
\]
so the minimizing levels are precisely those in
\eqref{eq:perturbed-selected-set}.  The final two assertions follow.
\end{proof}

Only active Dirichlet levels enter this comparison.  Levels absent from
\(u_0\) remain absent by \Cref{prop:spectral-support-conservation}.  If the
whole active spectrum is already contained in \(S_*\), the separation
condition is vacuous, although \eqref{eq:perturbed-selected-set} may still
split the degeneracy inside \(S_*\).

The restriction \(\deg R\leq\degree\) is what keeps the quotient in
\eqref{eq:perturbation-ratio} bounded at high spectral levels.  For
\(\deg R>\degree\) it may diverge.  The positivity of the leading
coefficient of \(\Ppoly_\varepsilon\) is automatic for small
\(|\varepsilon|\) when \(\deg R\leq\degree\).

\section{Polyharmonic and Swift--Hohenberg spectral consequences}
\label{sec:examples}

Two choices of \(\Ppoly\) make the difference in spectral ordering explicit.
The monomial preserves the order of the Dirichlet eigenvalues; the shifted
square orders them by distance from an interior level.

\subsection{The polyharmonic flow}
\label{subsec:polyharmonic-selection}

Let
\[
    \Ppoly_m(s)=s^m.
\]
Then \(\Ppoly_m(\Bop)=\Bop^m\) with domain \(\Dom(\Bop^m)\).  On a smooth domain this is the Navier realization of \(({-\Delta})^m\), with
\[
    u=\Delta u=\cdots=\Delta^{m-1}u=0
    \qquad\text{on }\partial\Om.
\]
The polynomial is strictly increasing on the positive Dirichlet spectrum, so its ordering agrees with the ordering of the active Dirichlet eigenvalues.

For an initial datum \(u_0\in\Sphere\), let
\begin{equation}
\label{eq:B-active-spectrum}
    \Sigma_{\Bop}(u_0)
    :=
    \left\{
        \lambda\in\spec(\Bop):
        \mathsf E^{\Bop}_{\lambda}u_0\neq0
    \right\},
\end{equation}
where
\(\mathsf E^{\Bop}_{\lambda}\) denotes the orthogonal projection onto
\[
    \Ker(\Bop-\lambda I).
\]
Since the Dirichlet spectrum is discrete and bounded below by
\(\lambda_1>0\), the set \(\Sigma_{\Bop}(u_0)\) has a smallest element.

\begin{corollary}[Polyharmonic spectral selection]
\label{cor:polyharmonic-selection}
Consider the linear constrained polyharmonic equation
\begin{equation}
\label{eq:linear-polyharmonic}
    u_t
    =
    -\Proj{u}\Bop^{\degree}u,
    \qquad
    u(0)=u_0\in\Sphere.
\end{equation}
Define
\begin{equation}
\label{eq:lambda-star-poly}
    \lambda_*(u_0)
    :=
    \min\Sigma_{\Bop}(u_0).
\end{equation}
Set
\begin{equation}
\label{eq:polyharmonic-selection-limit}
    u_\infty
    :=
    \frac{
        \mathsf E^{\Bop}_{\lambda_*(u_0)}u_0
    }{
        \norm{
            \mathsf E^{\Bop}_{\lambda_*(u_0)}u_0
        }_{\Hilb}
    }.
\end{equation}
Then
\[
    u(t)\longrightarrow u_\infty
    \qquad\text{strongly in }\Hilb.
\]

If
\[
    \Sigma_{\Bop}(u_0)
    \setminus
    \{\lambda_*(u_0)\}
    \neq\varnothing,
\]
set
\begin{equation}
\label{eq:polyharmonic-gap}
\begin{aligned}
    \delta_{\degree}(u_0)
    :=
    \min_{\substack{
        \lambda\in\Sigma_{\Bop}(u_0)\\
        \lambda>\lambda_*(u_0)
    }}
    \left(
        \lambda^{\degree}
        -
        \lambda_*(u_0)^{\degree}
    \right).
\end{aligned}
\end{equation}
Then
\[
    \delta_{\degree}(u_0)>0,
\]
and
\begin{equation}
\label{eq:polyharmonic-selection-rate}
    \norm{u(t)-u_\infty}_{\Hilb}
    \leq
    C(u_0)
    e^{-\delta_{\degree}(u_0)t}.
\end{equation}
If
\(u_0\in\Dom(\Bop^{\degree/2})\), the convergence also holds in
\(\Dom(\Bop^{\degree/2})\), with the same exponential spectral gap up to a
multiplicative constant.
\end{corollary}

\begin{proof}
For
\[
    \Ppoly_{\degree}(s)=s^{\degree},
\]
the function
\[
    s\longmapsto s^{\degree}
\]
is strictly increasing on \((0,\infty)\).  Since every Dirichlet eigenvalue
is positive, minimizing
\[
    \Ppoly_{\degree}(\lambda)
    =
    \lambda^{\degree}
\]
over the active spectral support is equivalent to minimizing \(\lambda\)
itself.  Hence the selected polynomial spectral value in
\Cref{thm:linear-spectral-selection} is
\[
    \mu_*(u_0)
    =
    \lambda_*(u_0)^{\degree},
\]
and its selected eigenspace is
\[
    \Ker(
        \Bop^{\degree}
        -
        \lambda_*(u_0)^{\degree}I
    )
    =
    \Ker(
        \Bop-\lambda_*(u_0)I
    ).
\]
The limit \eqref{eq:polyharmonic-selection-limit} therefore follows from
\eqref{eq:linear-limit}.

The active polynomial spectral gap from
\eqref{eq:active-spectral-gap} is precisely
\[
    \min_{\lambda>\lambda_*}
    \left(
        \lambda^{\degree}
        -
        \lambda_*^{\degree}
    \right)
\]
over the active Dirichlet spectrum, which is
\eqref{eq:polyharmonic-gap}.  The exponential estimate follows from
\eqref{eq:H-spectral-selection-rate}.  If
\[
    u_0\in\Dom(\Bop^{\degree/2}),
\]
then
\[
    \Dom(\Bop^{\degree/2})
    =
    \Dom(\Aop^{1/2})
\]
with equivalent norms, so
\eqref{eq:fractional-spectral-rate} with
\(\theta=1/2\) yields convergence in the energy topology.
\end{proof}

\begin{remark}
\label{rem:polyharmonic-first-mode}
If
\[
    \mathsf E^{\Bop}_{\lambda_1}u_0\neq0,
\]
then
\[
    \lambda_*(u_0)=\lambda_1,
\]
and the linear polyharmonic flow selects the normalized projection of the
initial datum onto the first Dirichlet eigenspace.  If the initial datum is
orthogonal to that eigenspace, the orthogonality is preserved and the flow
selects the lowest Dirichlet eigenspace that is present in the initial
spectral support.
\end{remark}

\subsection{The Swift--Hohenberg-type flow}
\label{subsec:SH-selection}

Fix \(\rhoSH>0\) and set
\[
    \Ppoly_{\rhoSH}(s)=(s-\rhoSH)^2.
\]
Then
\[
    \Ppoly_{\rhoSH}(\Bop)=(\Bop-\rhoSH I)^2
    =\Bop^2-2\rhoSH\Bop+\rhoSH^2I,
\]
with domain \(\Dom(\Bop^2)\).  Since \(\Bop=-\Delta_D\), its differential expression is \((\Delta+\rhoSH)^2\), and on a smooth domain the spectral realization carries the Navier boundary conditions
\[
    u=\Delta u=0
    \qquad\text{on }\partial\Om.
\]
The quadratic polynomial
\[
    s\longmapsto(s-\rhoSH)^2
\]
orders Dirichlet modes according to their distance from the spectral level
\(\rhoSH\).  This differs essentially from the monotone polyharmonic case.

For \(u_0\in\Sphere\), retain the active Dirichlet spectrum
\(\Sigma_{\Bop}(u_0)\) from
\eqref{eq:B-active-spectrum} and define
\begin{equation}
\label{eq:rho-distance}
    d_{\rhoSH}(u_0)
    :=
    \min_{\lambda\in\Sigma_{\Bop}(u_0)}
    |\lambda-\rhoSH|.
\end{equation}
Let
\begin{equation}
\label{eq:rho-selected-spectrum}
    \Sigma_{\rhoSH,*}(u_0)
    :=
    \left\{
        \lambda\in\Sigma_{\Bop}(u_0):
        |\lambda-\rhoSH|
        =
        d_{\rhoSH}(u_0)
    \right\},
\end{equation}
and define the selected subspace
\begin{equation}
\label{eq:rho-selected-space}
    \mathcal H_{\rhoSH,*}(u_0)
    :=
    \bigoplus_{
        \lambda\in\Sigma_{\rhoSH,*}(u_0)
    }
    \Ker(\Bop-\lambda I).
\end{equation}
Denote by
\[
    \mathsf E_{\rhoSH,*}
\]
the orthogonal projection onto
\(\mathcal H_{\rhoSH,*}(u_0)\).

\begin{theorem}[Swift--Hohenberg spectral selection]
\label{thm:SH-spectral-selection}
Consider the linear constrained problem
\begin{equation}
\label{eq:linear-SH}
    u_t
    =
    -\Proj{u}
    (\Bop-\rhoSH I)^2u,
    \qquad
    u(0)=u_0\in\Sphere.
\end{equation}
Then
\begin{equation}
\label{eq:linear-SH-explicit}
    u(t)
    =
    \frac{
        e^{-t(\Bop-\rhoSH I)^2}u_0
    }{
        \norm{
            e^{-t(\Bop-\rhoSH I)^2}u_0
        }_{\Hilb}
    },
\end{equation}
and
\begin{equation}
\label{eq:SH-selection-limit}
    u(t)
    \longrightarrow
    u_{\rhoSH,*}
    :=
    \frac{
        \mathsf E_{\rhoSH,*}u_0
    }{
        \norm{
            \mathsf E_{\rhoSH,*}u_0
        }_{\Hilb}
    }
    \qquad
    \text{strongly in }\Hilb.
\end{equation}

If
\[
    \Sigma_{\Bop}(u_0)
    \setminus
    \Sigma_{\rhoSH,*}(u_0)
    \neq\varnothing,
\]
define
\begin{equation}
\label{eq:SH-active-gap}
\begin{aligned}
    \delta_{\rhoSH}(u_0)
    &:=
    \min_{\lambda\in
        \Sigma_{\Bop}(u_0)
        \setminus
        \Sigma_{\rhoSH,*}(u_0)}
    \left[
        (\lambda-\rhoSH)^2
        -
        d_{\rhoSH}(u_0)^2
    \right].
\end{aligned}
\end{equation}
Then
\[
    \delta_{\rhoSH}(u_0)>0
\]
and
\begin{equation}
\label{eq:SH-selection-rate}
    \norm{
        u(t)-u_{\rhoSH,*}
    }_{\Hilb}
    \leq
    C(u_0)e^{-\delta_{\rhoSH}(u_0)t}.
\end{equation}
For every \(\tau>0\), the convergence also holds in \(\Dom(\Bop)\) for
\(t\geq\tau\), with an exponential estimate governed by the same active
spectral gap.  If \(u_0\in\Dom(\Bop)\), the corresponding estimate is valid
from \(t=0\).
\end{theorem}

\begin{proof}
For
\[
    \Ppoly_{\rhoSH}(s)
    =
    (s-\rhoSH)^2,
\]
\Cref{thm:normalized-semigroup} gives
\eqref{eq:linear-SH-explicit}.  The active polynomial spectral values are
\[
    (\lambda-\rhoSH)^2,
    \qquad
    \lambda\in\Sigma_{\Bop}(u_0).
\]
Their minimum is
\[
    d_{\rhoSH}(u_0)^2.
\]
Thus the eigenspaces of
\((\Bop-\rhoSH I)^2\) contributing to the selected spectral value are
exactly the Dirichlet eigenspaces satisfying
\[
    |\lambda-\rhoSH|
    =
    d_{\rhoSH}(u_0).
\]
Their orthogonal direct sum is
\(\mathcal H_{\rhoSH,*}(u_0)\), and the spectral projection of \(u_0\)
onto the minimizing eigenspace of
\((\Bop-\rhoSH I)^2\) is precisely
\[
    \mathsf E_{\rhoSH,*}u_0.
\]
The convergence \eqref{eq:SH-selection-limit} therefore follows from
\Cref{thm:linear-spectral-selection}.

If active modes remain outside the selected subspace, the next active
polynomial spectral value is strictly larger than
\(d_{\rhoSH}(u_0)^2\).  Since the active polynomial spectrum is discrete
and tends to \(+\infty\), the minimum in
\eqref{eq:SH-active-gap} exists and is strictly positive.  This is exactly
the active spectral gap of
\Cref{thm:linear-spectral-selection}, so
\eqref{eq:SH-selection-rate} follows from
\eqref{eq:H-spectral-selection-rate}.

Finally, for the quadratic polynomial,
\[
    \Dom(
        \Aop^{1/2}
    )
    =
    \Dom(\Bop).
\]
Thus if \(u_0\in\Dom(\Bop)\),
\eqref{eq:fractional-spectral-rate} with
\(\theta=1/2\), together with equivalence of the corresponding graph norms,
gives exponential convergence in \(\Dom(\Bop)\).
\end{proof}

\begin{corollary}[Resonant selection]
\label{cor:SH-resonance}
Suppose
\[
    \rhoSH\in\spec(\Bop)
\]
and
\[
    \mathsf E^{\Bop}_{\rhoSH}u_0\neq0.
\]
Then
\[
    d_{\rhoSH}(u_0)=0
\]
and the linear constrained flow satisfies
\begin{equation}
\label{eq:SH-resonant-limit}
    u(t)
    \longrightarrow
    \frac{
        \mathsf E^{\Bop}_{\rhoSH}u_0
    }{
        \norm{
            \mathsf E^{\Bop}_{\rhoSH}u_0
        }_{\Hilb}
    }.
\end{equation}
\end{corollary}

\begin{proof}
Since
\[
    \rhoSH\in\Sigma_{\Bop}(u_0),
\]
one has
\[
    |\rhoSH-\rhoSH|=0.
\]
No smaller nonnegative distance is possible, so
\[
    d_{\rhoSH}(u_0)=0.
\]
The only Dirichlet spectral value at zero distance from \(\rhoSH\) is
\(\rhoSH\) itself.  Hence
\[
    \mathcal H_{\rhoSH,*}(u_0)
    =
    \Ker(\Bop-\rhoSH I),
\]
and \eqref{eq:SH-resonant-limit} follows from
\Cref{thm:SH-spectral-selection}.
\end{proof}

The selected space need not correspond to a single Dirichlet eigenvalue.

\begin{proposition}[Degenerate selection across distinct spectral levels]
\label{prop:SH-two-level-degeneracy}
Suppose there exist distinct active eigenvalues
\[
    \lambda_-<\lambda_+
\]
such that
\begin{equation}
\label{eq:SH-midpoint}
    \rhoSH
    =
    \frac{\lambda_-+\lambda_+}{2},
\end{equation}
and suppose that no active Dirichlet eigenvalue lies closer to \(\rhoSH\).
Then
\[
    \lambda_-,
    \lambda_+
    \in
    \Sigma_{\rhoSH,*}(u_0),
\]
and
\begin{equation}
\label{eq:SH-two-level-space}
\begin{aligned}
    \Ker(\Bop-\lambda_-I)
    \oplus
    \Ker(\Bop-\lambda_+I)
    \subset
    \mathcal H_{\rhoSH,*}(u_0).
\end{aligned}
\end{equation}
If these are the only active eigenvalues at minimal distance from
\(\rhoSH\), then
\begin{equation}
\label{eq:SH-two-level-limit}
\begin{aligned}
    u_{\rhoSH,*}
    =
    \frac{
        \left(
            \mathsf E^{\Bop}_{\lambda_-}
            +
            \mathsf E^{\Bop}_{\lambda_+}
        \right)u_0
    }{
        \norm{
            \left(
                \mathsf E^{\Bop}_{\lambda_-}
                +
                \mathsf E^{\Bop}_{\lambda_+}
            \right)u_0
        }_{\Hilb}
    }.
\end{aligned}
\end{equation}
\end{proposition}

\begin{proof}
Condition \eqref{eq:SH-midpoint} gives
\[
    \rhoSH-\lambda_-
    =
    \lambda_+-\rhoSH
    =
    \frac{\lambda_+-\lambda_-}{2}.
\]
Therefore
\[
    |\lambda_--\rhoSH|
    =
    |\lambda_+-\rhoSH|.
\]
By assumption, no active eigenvalue has smaller distance from
\(\rhoSH\).  Thus both \(\lambda_-\) and \(\lambda_+\) belong to
\(\Sigma_{\rhoSH,*}(u_0)\), proving
\eqref{eq:SH-two-level-space}.  If no further active eigenvalue attains
this distance, the selected orthogonal projection is
\[
    \mathsf E_{\rhoSH,*}
    =
    \mathsf E^{\Bop}_{\lambda_-}
    +
    \mathsf E^{\Bop}_{\lambda_+}.
\]
Substitution into \eqref{eq:SH-selection-limit} gives
\eqref{eq:SH-two-level-limit}.
\end{proof}

\begin{theorem}[Generic uniqueness of the selected Dirichlet level]
\label{thm:SH-genericity}
Let
\begin{equation}
\label{eq:SH-exceptional-midpoints}
    \mathcal R_{\Bop}
    :=
    \left\{
        \frac{\lambda+\mu}{2}:
        \lambda,\mu\in\spec(\Bop),\ \lambda<\mu
    \right\}.
\end{equation}
Then \(\mathcal R_{\Bop}\) is countable.  If
\(\rhoSH\notin\mathcal R_{\Bop}\), then for every
\(u_0\in\Sphere\) the selected active Dirichlet spectrum
\(\Sigma_{\rhoSH,*}(u_0)\) contains exactly one Dirichlet eigenvalue.
The corresponding eigenspace may still have geometric multiplicity greater
than one.

For a fixed initial datum \(u_0\), if
\(\rho_0\notin\mathcal R_{\Bop}\), then there exists \(\eta>0\) such that
for all \(\rho\) with \(|\rho-\rho_0|<\eta\), the selected active Dirichlet
level is the same as at \(\rho_0\).  Hence the selected level is locally
constant in \(\rho\) away from the Dirichlet midpoint set and can change only
when a midpoint is crossed.
\end{theorem}

\begin{proof}
The Dirichlet spectrum is countable, so the set of pairwise midpoints in
\eqref{eq:SH-exceptional-midpoints} is countable.  Suppose that two distinct
active eigenvalues \(\lambda<\mu\) satisfy
\[
    |\lambda-\rhoSH|=|\mu-\rhoSH|.
\]
Squaring and subtracting gives
\[
    (\lambda-\rhoSH)^2-(\mu-\rhoSH)^2
    =
    (\lambda-\mu)(\lambda+\mu-2\rhoSH)=0.
\]
Since \(\lambda\neq\mu\), necessarily
\[
    \rhoSH=\frac{\lambda+\mu}{2}\in\mathcal R_{\Bop}.
\]
Thus, if \(\rhoSH\notin\mathcal R_{\Bop}\), two distinct active Dirichlet
levels cannot be equidistant from \(\rhoSH\).  Since the minimum distance in
\eqref{eq:rho-distance} is attained, the selected active Dirichlet level is
unique.

Fix now \(u_0\) and \(\rho_0\notin\mathcal R_{\Bop}\), and let
\(\lambda_*\) be the unique selected active level at \(\rho_0\).  There is a
strict positive gap between its squared distance and the next active squared
distance.  For
\(\lambda\in\Sigma_{\Bop}(u_0)\setminus\{\lambda_*\}\), define
\[
    G_\lambda(\rho)
    :=
    (\lambda-\rho)^2-(\lambda_*-\rho)^2.
\]
At \(\rho=\rho_0\), all these numbers are positive.  The difference
factorizes as
\[
    G_\lambda(\rho)
    =
    (\lambda-\lambda_*)
    (\lambda+\lambda_*-2\rho).
\]
For \(|\rho-\rho_0|\leq1\),
\[
    \lambda+\lambda_*-2\rho
    \geq
    \lambda+\lambda_*-2(|\rho_0|+1),
\]
so \(G_\lambda(\rho)\to+\infty\) uniformly on this neighborhood as
\(\lambda\to\infty\).  Hence all sufficiently large active levels remain
uniformly separated from zero there.  Only finitely many active levels remain
to be considered, and continuity of their functions \(G_\lambda\) yields a
common \(\eta>0\) for which
\[
    G_\lambda(\rho)>0
\]
for every active \(\lambda\neq\lambda_*\) whenever
\(|\rho-\rho_0|<\eta\).  Hence \(\lambda_*\) remains the unique selected
active level throughout that neighborhood.
\end{proof}

There are two sources of degeneracy in this example.  A Dirichlet eigenvalue
may itself have multiplicity greater than one, or two distinct eigenvalues
may lie at the same distance from \(\rho\).  The second case occurs only at
the midpoint set \(\mathcal R_{\Bop}\).  Away from that set, changing
\(\rho\) leaves the selected Dirichlet level unchanged locally; when a
midpoint is crossed, the preferred level may switch.


\section{Concluding remarks}
\label{sec:conclusion}

For the constrained linear equation, the polynomial acts through the numbers
\(\Ppoly(\lambda_j)\).  The selected state is the normalized projection onto
the active spectral subspace where these numbers are smallest, and the next
active value determines the exponential rate.  The same rate holds in every
fractional domain after positive time.  The design result shows that a finite
preferred spectral set can be prescribed, while the perturbation theorem
gives a quantitative condition under which an isolated selection is
unchanged.  In the Swift--Hohenberg case the only cross-level degeneracies are
Dirichlet midpoints.

The argument uses the invariance of the spectral components of the linear
flow.  A pointwise reaction term couples those components, so the present
mode-by-mode argument is no longer available.  A nonlinear problem in the
same polynomial Dirichlet framework is considered in
\cite{BrzezniakHussainNonlinearDynamics2026}.

\section*{Conflict of interest}
The author declares that he has no conflict of interest.

\section*{Data availability}
No datasets were generated or analysed in the present study.

\end{document}